\documentclass[11pt,reqno, A4paper]{amsart}
\usepackage{graphicx}
\usepackage{amssymb}
\usepackage{amsmath}
\usepackage{amsfonts}
\usepackage{color}
\usepackage{hyperref}
\usepackage{float}

\hypersetup{
colorlinks=true,
linkcolor=blue,
anchorcolor=blue,
citecolor=blue
}

\newtheorem{theorem}{Theorem}[section]
\newtheorem{lemma}[theorem]{Lemma}

\newtheorem{corollary}[theorem]{Corollary}

\theoremstyle{definition}
\newtheorem{definition}[theorem]{Definition}
\newtheorem{example}[theorem]{Example}

\newtheorem{remark}[theorem]{Remark}
\numberwithin{equation}{section}

\newenvironment{acknowledgement}{\smallskip{\sc Acknowledgments.}\rm}{\smallskip}

\allowdisplaybreaks
\def\mbox{\textrm}
\input{tcilatex}

\def\mbox{\textrm}
\def\Qlb#1{#1}

\def\Qcb#1{#1} 

\def\FRAME#1#2#3#4#5#6#7#8
{
 \begin{figure}[H]
 \begin{center}
\includegraphics[width=#2,height=#3]{#7}
 \caption{#5}
 \label{#6}
 \end{center}
 \end{figure}
}

\begin{document}
\title[Torsion of digraphs]{Torsion of digraphs and path complexes}
\date{December 2020}
\author[A. Grigor'yan]{Alexander Grigor'yan}
\address{Alexander Grigor'yan, Department of Mathematics, University of
Bielefeld, 33501 Bielefeld, Germany}
\email{grigor@math.uni-bielefeld.de}
\author[Y. Lin]{Yong Lin}
\address{Yong Lin, Yau Mathematical Sciences Center, Tsinghua University,
Beijing, 100084, China.}
\email{yonglin@tsinghua.edu.cn}
\author[S.-T. Yau]{Shing-Tung Yau}
\address{Shing-Tung Yau, Department of Mathematics, Harvard University,
Cambridge, Massachusetts, USA.}
\email{yau@math.harvard.edu}
\thanks{AG is supported by the SFB1283 of the German Research Council. YL is supported by the National Science Foundation of China
 (Grant No. 12071245 and 11761131002)}

\begin{abstract}
We define the notions of Reidemeister torsion and analytic torsion for
directed graphs by means of the path homology theory introduced by the
authors in \cite{Grigoryan-Lin-Muranov-Yau2013,
Grigoryan-Lin-Muranov-Yau2014, Grigoryan-Lin-Muranov-Yau2015,
Grigoryan-Lin-Muranov-Yau2020}. We prove the identity of the two notions of
torsions as well as obtain formulas for torsions of Cartesian products and
joins of digraphs.
\end{abstract}

\maketitle
\tableofcontents



\section{Introduction}

\label{S:intr} \setcounter{equation}{0}Let $M$ be a compact oriented
Riemannian manifold. Assume that $M$ is triangulated by a simplicial complex 
$K$. Let $\rho $ be a acyclic representation of $\pi _{1}(K)$ by orthogonal
matrices, i.e., the twisted cohomology group $H^{p}(K;\rho )$ is trivial for
all $p$. The Reidemeister torsion $\tau _{\rho }(M)$ is defined from the
cochain complex of $K$ by taking a alternating product of determinants \cite%
{Reidemeister1935, Milnor1966}. It is a manifold invariant and is used to
distinguish homotopy equivalent spaces \cite{Cohen1973}.

To describe the Reidemeister torsion in analytic terms, Ray and Singer \cite%
{Ray-Singer1971} defined an analytic torsion $T_{\rho }(M)$ for any compact
oriented manifold $M$ and orthogonal representation $\rho $ of the
fundamental group $\pi _{1}(M)$. Their definition used the spectrum of the
Hodge Laplacian on twisted forms. When $\rho $ is acyclic and orthogonal,
Cheeger \cite{Cheeger1979} and M\"{u}ller \cite{Muller1978} proved that $%
\tau _{\rho }(M)=T_{\rho }(M)$. When $\rho $ is orthogonal but not acyclic,
one still can define the analytic torsion $T_{\rho }(M)$ \cite{Fried1986}.

In this paper we introduce the notions of Reidemeister and analytic torsions
on finite digraphs by means of the path homology theory of Grigoryan, Lin,
Muranov and Yau \cite{Grigoryan-Lin-Muranov-Yau2013}, \cite%
{Grigoryan-Lin-Muranov-Yau2014}, \cite{Grigoryan-Lin-Muranov-Yau2015}, \cite%
{Grigoryan-Lin-Muranov-Yau2020}. Namely, we use the homology basis to
construct a preferred basis of the path complex on a digraph $G$, which
leads to the definition of the Reidemeister torsion $\tau (G)$. Next, we
define the Hodge Laplace operator $\Delta _{p}$ acting on $p$-paths and use
the positive eigenvalues of $\Delta _{p}$ in order to define the analytic
torsion $T(G)$. Although the homology groups can be nontrivial in our case,
we still can prove that $\tau (G)=T(G)$ (Theorem \ref{T:main_thm}) by using
an extension of the argument of \cite[Proposition 1.7]{Ray-Singer1971}.

Given two finite digraphs $X$ and $Y$, we obtain formulas for the torsions
of their Cartesian product $X\Box Y$ and join $X\ast Y$ (Theorems \ref%
{T:main_thm1} and \ref{Tmain2}). Our proofs rely essentially on the K\"{u}%
nneth formulas for chain complexes of $X\Box Y$ and $X\ast Y$ proved in \cite%
{GMY} and \cite{Grigoryan-Lin-Muranov-Yau2020}. The approach to the proof is
borrowed from \cite[Thm. 2.5]{Ray-Singer1971} but our setting is more
complicated in the following sense. The notion of torsion depends on the
choice of an inner product in the chain spaces, and the cases of the
Cartesian product and join require usage of different inner products.
Besides, the case of join requires usage of an augmented chain complex. For
that reason, the final formulas for $\tau \left( X\Box Y\right) $ and $\tau
\left( X\ast Y\right) $ stated in Corollaries \ref{Cor1} and \ref{Cor2} are
more complicated than one could expect.

In Section \ref{Sec1} we revise the path homology theory. In Section \ref%
{SecHodge} we introduce the notions of the Hodge Laplacian on an arbitrary
finite-dimensional chain complex, prove the Hodge decomposition, define the
notions of R-torsion and analytic torsion, and prove their identity (Theorem %
\ref{T:main_thm}).

In Section \ref{SecP} we revise the notion of the Cartesian product of
digraphs, the K\"{u}nneth formula for the Cartesian product, and use it to
prove the formula for the torsion of $X\Box Y$ (Theorem \ref{T:main_thm1}
and Corollaries \ref{Cor1n}, \ref{Cor1}). In Section \ref{SecJ} we fulfil a
similar program for the join of digraphs (Theorem \ref{Tmain2} and
Corollaries \ref{Cor2}, \ref{Cor2n}).

We give numerous examples of application of our results by computing
torsions of various digraphs including simplices, cubes, spheres, cycles,
prism, etc.

\section{Path complexes and path homology}

\label{Sec1}\label{S:rei_tor_and_ana_tor} \setcounter{equation}{0}Let us
briefly revise the definition of \emph{path complex} and \emph{path homology}
introduced by Grigoryan, Lin, Muranov and Yau in \cite%
{Grigoryan-Lin-Muranov-Yau2020} (see also \cite%
{Grigoryan-Lin-Muranov-Yau2013}).

\subsection{Path complex}

\label{SecPC}Let $V$ be a finite set. For any $p\geq 0$, an \emph{elementary}
$p$-\emph{path} is any (ordered) sequence $i_{0},\ldots ,i_{p}$ of $p+1$
vertices of $V$ that will be denoted simply by $i_{0}\ldots i_{p}$ or by $%
e_{i_{0}\ldots i_{p}}$. The number $p$ is called the length of the path $%
i_{0}...i_{p}.$

Formal $\mathbb{R}$-linear combinations of $e_{i_{0}...i_{p}}$ are called $p$%
-\emph{paths}. Denote by $\Lambda _{p}=\Lambda _{p}\left( V\right) $ the
linear space of all $p$-paths; that is, the elements of $\Lambda _{p}$ are 
\begin{equation*}
v=\sum_{i_{0}i_{1}\ldots i_{p}}v^{_{i_{0}i_{1}\ldots
i_{p}}}e_{i_{0}i_{1}...i_{p}},\ \ \text{where}\ \ v^{_{i_{0}i_{1}\ldots
i_{p}}}\in \mathbb{R}.
\end{equation*}

\begin{definition}
\label{D:defi_of_Boundary oper}For any $p\geq 0,$ define the \emph{boundary
operator} $\partial :\Lambda _{p+1}\rightarrow \Lambda _{p}$ by 
\begin{equation}
\left( \partial v\right) ^{i_{0}\ldots i_{p}}=\sum\limits_{k\in
V}\sum\limits_{q=0}^{p+1}\left( -1\right) ^{q}v^{i_{0}\ldots
i_{q-1}ki_{q}\ldots i_{p}},  \label{e:dv}
\end{equation}%
where the index $k$ is inserted so that it is preceded by $q$ indices $%
i_{0}...i_{q-1}.$ Set also $\Lambda _{-1}=\left\{ 0\right\} $ and define the
operator $\partial :\Lambda _{0}\rightarrow \Lambda _{-1}$ by setting $%
\partial v=0$ for all $v\in \Lambda _{0}$.
\end{definition}

It follows from (\ref{e:dv}) that 
\begin{equation}
\partial e_{j_{0}\ldots j_{p+1}}=\sum\limits_{q=0}^{p+1}\left( -1\right)
^{q}e_{j_{0}\ldots \widehat{j_{q}}\ldots j_{p+1}},  \label{3.4}
\end{equation}%
where $\widehat{\cdot }$ means omission of the index.

It is easy to show that $\partial ^{2}v=0$ for any $v\in \Lambda _{p}$ (\cite%
[Lemma 2.1]{Grigoryan-Lin-Muranov-Yau2020}). Hence, the family of linear
spaces $\left\{ \Lambda _{p}\right\} $ with the boundary operator $\partial $
determine a chain complex that will be denoted by $\Lambda \left( V\right) $.

\begin{definition}
\label{D:defi_of_Regular path} An elementary $p$-path $e_{i_{0}\ldots i_{p}}$
on a set $V$ is called \emph{regular} if $i_{k}\neq i_{k+1}$ for all $%
k=0,\ldots ,p-1$, and \emph{irregular} otherwise.
\end{definition}

Let $I_{p}$ be the subspace of $\Lambda _{p}$ that is spanned by all
irregular $e_{i_{0}\ldots i_{p}}.$ It is easy to verify that $\partial
I_{p}\subset I_{p-1}$ (cf. \cite{Grigoryan-Lin-Muranov-Yau2020}). Hence, the
boundary operator $\partial $ is well-defined on the quotient space $%
\mathcal{R}_{p}:=\Lambda _{p}/I_{p}$: 
\begin{equation*}
\partial :\mathcal{R}_{p}\rightarrow \mathcal{R}_{p-1}
\end{equation*}%
for all $p\geq 0.$ Clearly, $\mathcal{R}_{p}$ is linearly isomorphic to the
space of all regular $p$-paths: 
\begin{equation}
\mathcal{R}_{p}\cong \mathrm{span}\left\{ e_{i_{0}\ldots i_{p}}:i_{0}\ldots
i_{p}\text{ is regular}\right\} .  \label{e:Rp=}
\end{equation}%
For simplicity of notation, we will identify $\mathcal{R}_{p}$ with the
space of all regular $p$-paths. With this identification, the formula (\ref%
{e:dv}) for the operator $\partial :\mathcal{R}_{p+1}\rightarrow \mathcal{R}%
_{p}$ is true only for regular paths $i_{0}...i_{p}$ whereas $\left(
\partial v\right) ^{i_{0}\ldots i_{p}}=0$ if $i_{0}...i_{p}$ is irregular.
The identity (\ref{3.4}) remains true if we replace by $0$ each irregular
path on the right hand side$.$

Denote by $\mathcal{R}\left( V\right) $ the chain complex $\left\{ \mathcal{R%
}_{p}\right\} $ with the boundary operator $\partial .$

\begin{definition}
A \emph{path complex} over a set $V$ is a non-empty collection $P$ of
regular elementary paths on $V$ with the following property: 
\begin{equation}
\text{if\ }\ e_{i_{0}...i_{n}}\in P\ \text{then\ }e_{i_{0}...i_{n-1}}\in P\ 
\text{and\ }e_{i_{1}...i_{n}}\in P.  \label{t}
\end{equation}
\end{definition}

When a path complex $P$ is fixed, all the paths from $P$ are called \emph{%
allowed}, whereas the elementary paths that are not in $P$ are called \emph{%
non-allowed}. Condition (\ref{t}) means that if we remove the first or the
last element of an allowed $n$-path then the resulting $\left( n-1\right) $%
-path is also allowed.

The set of all $n$-paths from $P$ is denoted by $P_{n}$. The set $P_{-1}$
consists of a single empty path $e$. The elements of $P_{0}$ (that is,
allowed $0$-paths) are called the \emph{vertices} of $P$. Clearly, $P_{0}$
is a subset of $V$. By the property (\ref{t}), if $i_{0}...i_{n}\in P$ then
all $i_{k}$ are vertices of $P$. Hence, we can (and will) remove from the
set $V$ all non-vertices so that $V=P_{0}.$

There are two natural families/examples of path complexes. Any abstract
finite simplicial complex $S$ is a collection of subsets of a finite vertex
set $V$ that satisfies the following property: 
\begin{equation*}
\text{if\ }\sigma \in S\ \text{then any subset of }\sigma \ \text{is also in 
}S.
\end{equation*}%
Let us enumerate the elements of $V$ by distinct reals and identify any
subset $s$ of $V$ with the elementary path that consists of the elements of $%
s$ put in the (strictly) increasing order. Denote by $P\left( S\right) $
this collections of elementary paths on $V$ that uniquely determines $S$.
The defining property of a simplex can be restated the following: 
\begin{equation}
\text{if \ }v\in P\left( S\right) \text{ then any subsequence of \ }v\text{
is also in \ }P\left( S\right) \text{.}  \label{s}
\end{equation}%
Consequently, the family $P\left( S\right) $ satisfies the property (\ref{t}%
) so that $P\left( S\right) $ is a path complex. The allowed $n$-paths in $%
P\left( S\right) $ are exactly the $n$-simplexes.

\subsection{Digraphs}

Another natural family of path complexes comes from digraphs.

\begin{definition}
\label{D:defi_of_Digraph} A \emph{digraph} $G=(V,E)$ is a couple, where $V$
is a set, whose elements are called the \emph{vertices}, and $E$ is a subset
of $\{V\times V\setminus \mathrm{diag}\}$ that consists of ordered pairs of
vertices called (directed) \emph{edges} or \emph{arrows}. The fact that a
pair $\left( x,y\right) $ is an arrow will be denoted by $x\rightarrow y.$
\end{definition}

An elementary $n$-path $i_{0}...i_{n}$ on the vertex set $V$ of a digraph is
called allowed if $i_{k-1}\rightarrow i_{k}$ for any $k=1,...,n$. Denote by $%
P_{n}=P_{n}\left( G\right) $ the set of all allowed $n$-paths. In
particular, we have $P_{0}=V$ and $P_{1}=E$. Clearly, the collection $%
P=\tbigcup_{n}P_{n}$ of all allowed paths satisfies the condition (\ref{t})
so that $P$ is a path complex. This path complex is naturally associated
with the digraph $G$ and will be denoted by $P\left( G\right) $.

\subsection{Path homology}

Let us return to an arbitrary path complex $P$ over $V$. Denote by $\mathcal{%
A}_{p}\left( P\right) $ the subspace of $\mathcal{R}_{p}\left( V\right) $
spanned by the allowed elementary $p$-paths, that is, 
\begin{equation}
\mathcal{A}_{p}=\mathrm{span}{}\left\{ e_{i_{0}\ldots i_{p}}:i_{0}\ldots
i_{p}\in E_{p}\right\} .  \label{e:Padef}
\end{equation}%
The elements of $\mathcal{A}_{p}$ are called \emph{allowed} $p$-paths.

Note that the spaces $\mathcal{A}_{p}$ of allowed paths are in general \emph{%
not} invariant for $\partial $. Consider the following subspace of $\mathcal{%
A}_{p}$ 
\begin{equation}
\Omega _{p}\equiv \Omega _{p}\left( P\right) :=\left\{ v\in \mathcal{A}%
_{p}:\partial v\in \mathcal{A}_{p-1}\right\} .  \label{e:PE=}
\end{equation}%
The spaces $\Omega _{p}$ are $\partial $-invariant. Indeed, $v\in \Omega
_{p} $ implies $\partial v\in \mathcal{A}_{p-1}$ and $\partial \left(
\partial v\right) =0\in \mathcal{A}_{p-2}$, whence $\partial v\in \Omega
_{p-1}$. The elements of $\Omega _{p}$ are called $\partial $-\emph{invariant%
} $p$-paths.

Hence, we obtain a chain complex $\Omega =\Omega \left( P\right) :$ 
\begin{equation}
\begin{array}{cccccccccccc}
0 & \leftarrow & \Omega _{0} & \overset{\partial }{\leftarrow } & \Omega _{1}
& \overset{\partial }{\leftarrow } & \cdots & \overset{\partial }{\leftarrow 
} & \Omega _{p-1} & \overset{\partial }{\leftarrow } & \Omega _{p} & \overset%
{\partial }{\leftarrow }\cdots%
\end{array}
\label{Omi}
\end{equation}%
By construction we have $\Omega _{0}=\mathcal{A}_{0}$ and $\Omega _{1}=%
\mathcal{A}_{1}$, while in general $\Omega _{p}\subset \mathcal{A}_{p}$.

Set 
\begin{equation*}
Z_{p}=\ker \partial |_{\Omega _{p}}\ \text{and}\ \ B_{p}=\partial \Omega
_{p+1}.
\end{equation*}

\begin{definition}
Define for all $p\geq 0$ the \emph{path homology groups }$H_{p}\left(
P\right) $ of the path complex $P$ by 
\begin{equation}
H_{p}\left( P\right) :=H_{p}\left( \Omega \left( P\right) \right)
=Z_{p}/B_{p}.  \label{HZB}
\end{equation}
\end{definition}

Let us note that the spaces $H_{p}\left( P\right) $ (as well as the spaces $%
\Omega _{p}\left( P\right) $) can be computed directly by definition using
simple tools of linear algebra, in particular, those implemented in modern
computational software. On the other hand, some theoretical tools for
computation of homology groups, like homotopy theory and K\"{u}nneth
formulas, were developed in \cite{Grigoryan-Lin-Muranov-Yau2014}, \cite{GMY}%
, \cite{Grigoryan-Lin-Muranov-Yau2020}.

In particular, for any digraph $G$ define its path homology groups by%
\begin{equation*}
H_{p}\left( G\right) =H_{p}\left( P\left( G\right) \right) .
\end{equation*}%
In what follows we are going to deal with only finite chain complexes: 
\begin{equation}
\begin{array}{cccccccccccccccc}
0 & \leftarrow & \Omega _{0} & \overset{\partial }{\leftarrow } & \Omega _{1}
& \overset{\partial }{\leftarrow } & \cdots & \overset{\partial }{\leftarrow 
} & \Omega _{p-1} & \overset{\partial }{\leftarrow } & \Omega _{p} & \overset%
{\partial }{\leftarrow }\cdots & \overset{\partial }{\leftarrow } & \Omega
_{N} & \leftarrow & 0%
\end{array}
\label{OmN}
\end{equation}%
where $N\in \mathbb{N}.$ Clearly, any chain complex (\ref{Omi}) can be
truncated to the form (\ref{OmN}).

For path complexes and digraphs this means that we restrict the length of
allowed paths to $N$. There is a large family of digraphs where the chain
complex $\Omega $ is finite naturally because $\Omega _{N}=\left\{ 0\right\} 
$ for some $N$ (and, hence, $\Omega _{n}=\left\{ 0\right\} $ for all $n\geq
N $). All examples of digraphs that are considered in this paper have
naturally finite chain complex $\Omega .$

If this is not the case then we can choose $N$ arbitrarily and truncate the
chain complex $\Omega $ to (\ref{OmN}). The number $N$ will be referred to
as the dimension of the chain complex (\ref{OmN}) or that of the underlying
path complex.

Some examples of chain complexes $\Omega $ and homology groups of digraphs
will be given in Section \ref{SecEx}.

\section{Finite chain complexes}

\label{SecHodge}Let us fix a finite chain complex $\Omega $ (\ref{OmN}) of
finite dimensional linear spaces $\Omega _{p}.$ We are interested in chain
complexes that are coming from path complexes as described above, but in
this section we revise rather well known facts about general chain complexes 
$\Omega .$

Let us choose arbitrarily an inner product $\langle ,\rangle $ in each
linear space $\Omega _{p}.$ In the case when $\Omega $ comes from a path
complex, an inner product in $\Omega _{p}$ can be taken from the ambient
space $\mathcal{R}_{p}.$ In this paper we use two different inner products
in $\mathcal{R}_{p}$. Let $u,v\in \mathcal{R}_{p}$ and 
\begin{equation*}
u=\sum_{\mathbf{i}}u^{\mathbf{i}}e_{\mathbf{i}}\ \ \ \text{and\ \ }v=\sum_{%
\mathbf{i}}v^{\mathbf{i}}e_{\mathbf{i}}
\end{equation*}%
where $\mathbf{i}=i_{0}...i_{p}$. The first (\emph{standard}) inner product
is 
\begin{equation}
\langle u,v\rangle =\sum_{\mathbf{i}}u^{\mathbf{i}}v^{\mathbf{i}},
\label{ii}
\end{equation}%
and the second (\emph{normalized}) inner product is 
\begin{equation}
\langle u,v\rangle =\frac{1}{p!}\sum_{\mathbf{i}}u^{\mathbf{i}}v^{\mathbf{i}%
}.  \label{iip}
\end{equation}%
These inner products will be used in examples and in Section \ref{SecP}, but
in general we do not impose any restriction on the choice of inner products
in the spaces $\Omega _{p}.$

\subsection{Hodge Laplacian}

Denote by $\partial _{p}$ the operator $\partial :\Omega _{p}\rightarrow
\Omega _{p-1}.$ Assuming that the inner product structure in $\Omega $ is
chosen, consider the operator $\partial _{p}^{\ast }:\Omega
_{p-1}\rightarrow \Omega _{p}$ that is the adjoint operator of $\partial
_{p} $ with respect to the inner products in $\Omega _{p}$ and $\Omega
_{p-1} $.

\begin{definition}
Define the \emph{Hodge-Laplace operator} $\Delta _{p}:\Omega _{p}\rightarrow
\Omega _{p}$ by 
\begin{equation}
\Delta _{p}u=\partial _{p}^{\ast }\partial _{p}u+\partial _{p+1}\partial
_{p+1}^{\ast }u.  \label{Dede}
\end{equation}
\end{definition}

We will use a shorter notation%
\begin{equation*}
\Delta _{p}u=\partial ^{\ast }\partial u+\partial \partial ^{\ast }u
\end{equation*}
since it is clear from this expression in which spaces $\Omega _{p}$ act the
operators $\partial $ and $\partial ^{\ast }.$

An element $u\in \Omega _{p}$ is called \emph{harmonic} if $\Delta _{p}u=0.$

\begin{lemma}
\label{Lemharm}An element $u\in \Omega _{p}$ is harmonic if and only if $%
\partial u=0$ and $\partial ^{\ast }u=0.$
\end{lemma}

\begin{proof}
If $\partial u=0$ and $\partial ^{\ast }u=0$ then by (\ref{Dede}) we have $%
\Delta _{p}u=0.$ Conversely, if $\Delta _{p}u=0$ then we obtain%
\begin{equation*}
0=\left\langle \Delta _{p}u,u\right\rangle =\left\langle \partial ^{\ast
}\partial u,u\right\rangle +\left\langle \partial \partial ^{\ast
}u,u\right\rangle =\left\langle \partial u,\partial u\right\rangle
+\left\langle \partial ^{\ast }u,\partial ^{\ast }u\right\rangle ,
\end{equation*}%
whence $\left\Vert \partial u\right\Vert =\left\Vert \partial ^{\ast
}u\right\Vert =0.$
\end{proof}

Denote by $\mathcal{H}_{p}$ the set of all harmonic elements in $\Omega _{p}$
so that $\mathcal{H}_{p}$ is a subspace of $\Omega _{p}.$

\begin{lemma}
\emph{(Hodge decomposition)} The space $\Omega _{p}$ is an orthogonal sum of
three subspaces as follows:%
\begin{equation}
\Omega _{p}=\partial \Omega _{p+1}\dbigoplus \partial ^{\ast }\Omega
_{p-1}\dbigoplus \mathcal{H}_{p}.  \label{Hd}
\end{equation}
\end{lemma}

\begin{proof}
If $u\in \partial \Omega _{p+1}$ and $v\in \partial ^{\ast }\Omega _{p-1}$
then $u=\partial u^{\prime }$ and $v=\partial ^{\ast }v^{\prime }$, and we
have%
\begin{equation*}
\langle u,v\rangle =\langle \partial u^{\prime },\partial ^{\ast }v^{\prime
}\rangle =\left\langle \partial ^{2}u^{\prime },v^{\prime }\right\rangle =0
\end{equation*}%
so that the subspaces $\partial \Omega _{p+1}$ and $\partial ^{\ast }\Omega
_{p-1}$ are orthogonal. Denote by $K$ the orthogonal complement of $\partial
\Omega _{p+1}\dbigoplus \partial ^{\ast }\Omega _{p-1}$ in $\Omega _{p}.$
Then we have%
\begin{equation*}
u\in K\Leftrightarrow \left\langle u,v\right\rangle =0\ \ \text{for all }%
v\in \partial \Omega _{p+1}\ \text{and }v\in \partial ^{\ast }\Omega _{p-1}
\end{equation*}%
that is, 
\begin{eqnarray*}
u\overset{}{\in }K &\Leftrightarrow &\left\langle u,\partial v^{\prime
}\right\rangle =0\ \ \forall v^{\prime }\in \Omega _{p+1}\ \ \text{and\ \ }%
\left\langle u,\partial ^{\ast }v\right\rangle \ \ \forall v\in \Omega _{p-1}
\\
&\Leftrightarrow &\left\langle \partial ^{\ast }u,v^{\prime }\right\rangle
=0\ \ \forall v^{\prime }\in \Omega _{p+1}\ \ \text{and\ \ }\left\langle
\partial u,v\right\rangle \ \ \forall v\in \Omega _{p-1} \\
&\Leftrightarrow &\partial ^{\ast }u=0\ \ \text{and\ \ }\partial u=0 \\
&\Leftrightarrow &u\in \mathcal{H}_{p}.
\end{eqnarray*}%
Hence, $K=\mathcal{H}_{p}$ which finishes the proof.
\end{proof}

\begin{corollary}
There is a natural isomorphism 
\begin{equation}
H_{p}\cong \mathcal{H}_{p}.  \label{HH}
\end{equation}
\end{corollary}

\begin{proof}
Observe first that \thinspace $Z_{p}:=\ker \partial _{p}$ is the orthogonal
complement of $\partial ^{\ast }\Omega _{p-1}$ in $\Omega _{p}$ because, for
any $u\in \Omega _{p}$,%
\begin{equation*}
u\in Z_{p}\Leftrightarrow \partial u=0\Leftrightarrow \left\langle
u,\partial ^{\ast }v\right\rangle =0\ \forall v\in \Omega
_{p-1}\Leftrightarrow u\bot \partial ^{\ast }\Omega _{p-1}.
\end{equation*}%
It follows from (\ref{Hd}) that 
\begin{equation}
Z_{p}=\partial \Omega _{p+1}\dbigoplus \mathcal{H}_{p}=B_{p}\dbigoplus 
\mathcal{H}_{p}  \label{Zp=}
\end{equation}%
whence 
\begin{equation*}
\mathcal{H}_{p}\cong Z_{p}/B_{p}=H_{p}.
\end{equation*}
\end{proof}

\begin{remark}
It follows from this argument that $\mathcal{H}_{p}$ is an orthogonal
complement of $B_{p}$ in $Z_{p}$ and that a harmonic form $u\in \mathcal{H}%
_{p}$ that corresponds to a homology class $\omega \in H_{p}$, minimizes the
norm $\left\Vert \cdot \right\Vert $ among all elements of $\omega .$
\end{remark}

\subsection{R-torsion}

\label{SecR}Let $\Omega $ be a finite chain complex of finite dimensional
linear spaces over $\mathbb{R}:$%
\begin{equation*}
0\overset{\partial }{\leftarrow }\Omega _{0}\overset{\partial }{\leftarrow }%
\Omega _{1}\overset{\partial }{\leftarrow }...\overset{\partial }{\leftarrow 
}\Omega _{p-1}\overset{\partial }{\leftarrow }\Omega _{p}\overset{\partial }{%
\leftarrow }...\overset{\partial }{\leftarrow }\Omega _{N}\overset{\partial }%
{\leftarrow }0.
\end{equation*}%
Denote $B_{p}=\partial \Omega _{p+1}$, $Z_{p}=\ker \partial |_{\Omega _{p}}$
and $H_{p}=Z_{p}/B_{p}.$

In any $\Omega _{p}$ choose a basis $\omega _{p}$ and a basis $h_{p}$ in $%
H_{p}$. For each element of $h_{p}$ choose its representative in $Z_{p}$ and
denote the resulting independent set by $\widetilde{h}_{p}.$

Let $b_{p}$ be \emph{any} basis in $B_{p}$. For each element $w\in b_{p-1}$
choose one element $v\in \partial ^{-1}w\subset \Omega _{p}$ so that $%
\partial v=w$. Let $\widetilde{b}_{p}$ be the collection of chosen elements $%
v$ so that%
\begin{equation}
b_{p-1}=\partial \widetilde{b}_{p}.  \label{bp-1}
\end{equation}%
Note that always $\widetilde{b}_{0}=\emptyset .$ Since $b_{p-1}$ is linearly
independent, the set $\widetilde{b}_{p}$ is also linearly independent.
Clearly, the union $(b_{p},\widetilde{h}_{p})$ is a basis in $Z_{p}$. Since
the subspaces $Z_{p}$ and $\limfunc{span}(\widetilde{b}_{p})$ of $\Omega
_{p} $ have a trivial intersection $\left\{ 0\right\} $, by the rank-nullity
theorem we conclude that the direct sum of these subspaces is $\Omega _{p}$.
Hence, the union $(b_{p},\widetilde{h}_{p},\widetilde{b}_{p})$ of the thee
sequences is a basis in $\Omega _{p}$.

If $U$ and $W$ are two bases in an $n$-dimensional linear space, then denote
by $\left( U/W\right) $ the transformation matrix from $W$ to $U$ and set 
\begin{equation*}
\lbrack U/W]=\left\vert \det \left( U/W\right) \right\vert .
\end{equation*}%
In the case $n=0$ set $[U/W]=1.$

Denote $\omega $ the collection $\left\{ \omega _{p}\right\} $ of the bases
in $\Omega _{p}$ and similarly let $h=\left\{ h_{p}\right\} $ be the
collection of the bases in $H_{p}.$

\begin{definition}
\label{D:defi_of_Rei_tor}The \emph{R-torsion} $\tau (\Omega ,\omega ,h)$ of
the chain complex $\Omega $ with the preferred bases $\omega $ and $h$ is a
positive real number defined by 
\begin{equation}
\log \tau (\Omega ,\omega ,h)=\sum_{p=0}^{N}(-1)^{p}\log [b_{p},\widetilde{h}%
_{p},\widetilde{b}_{p}\,/\,\omega _{p}].  \label{tau}
\end{equation}
\end{definition}

We justify this definition in the following statement.

\begin{lemma}
\label{LemR}

\begin{itemize}
\item[$\left( a\right) $] The value of $\tau \left( \Omega ,\omega ,h\right) 
$ does not depend on the choice of the bases $b_{p}$, the representatives in 
$\widetilde{b}_{p}$ and the representatives in $\widetilde{h}_{p}$ (which
justifies the notation $\tau (\Omega ,\omega ,h)$).

\item[$\left( b\right) $] If $\omega ^{\prime }$ and $h^{\prime }$ are other
collections of bases in $\Omega $ and $H$ respectively, then 
\begin{equation}
\log \tau \left( \Omega ,\omega ^{\prime },h^{\prime }\right) =\log \tau
\left( \Omega ,\omega ,h\right) +\sum_{p=0}^{N}\left( -1\right) ^{p}\left(
\log \left[ \omega _{p}/\omega _{p}^{\prime }\right] +\log \left[
h_{p}^{\prime }/h_{p}\right] \right) .  \label{log+}
\end{equation}
\end{itemize}
\end{lemma}

The relation $\left[ U/W\right] =1$ for bases $U$ and $W$ is an equivalence
relation, and each equivalence class determines a \emph{volume form} in the
underlying linear space. We see from (\ref{log+}) that $\tau \left( \Omega
,\omega ,h\right) $ depends only on the volume forms determined by $\omega $
and $h$ in the spaces $\Omega _{p}$ and $H_{p}$, respectively,

\begin{proof}[Proof of Lemma \protect\ref{LemR}]
$\left( a\right) $ Let $b_{p}^{\prime }$ be another basis in $B_{p}$ with
the corresponding set $\widetilde{b}_{p}^{\prime }$, and $\widetilde{h}%
_{p}^{\prime }$ be another set of representatives of $h_{p}$. Let us first
verify that 
\begin{equation}
\lbrack b_{p}^{\prime },\widetilde{h}_{p}^{\prime },\widetilde{b}%
_{p}^{\prime }\,/\,b_{p},\widetilde{h}_{p},\widetilde{b}_{p}]=[b_{p}^{\prime
}/b_{p}][{b}_{p-1}^{\prime }/{b}_{p-1}].  \label{bhb}
\end{equation}%
Let $\widetilde{h}_{p}=\left\{ u_{1},u_{2},...\right\} $ and $\widetilde{h}%
_{p}^{\prime }=\left\{ u_{1}^{\prime },u_{2}^{\prime },...\right\} .$ Since $%
u_{i}^{\prime }$ and $u_{i}$ represent the same homology class, we have 
\begin{equation}
u_{i}^{\prime }=u_{i}+b_{i}\ \ \text{for some }b_{i}\in B_{p}.  \label{u'}
\end{equation}%
Let $\widetilde{b}_{p}=\left\{ v_{1},v_{2},...\right\} $ and $\widetilde{b}%
_{p}^{\prime }=\left\{ v_{1}^{\prime },v_{2}^{\prime },...\right\} $ so that 
\begin{equation*}
b_{p-1}=\partial \widetilde{b}_{p}=\left\{ \partial v_{1},\partial
v_{2},...\right\} \ \ \text{and\ \ \ }b_{p-1}^{\prime }=\partial \widetilde{b%
}_{p}^{\prime }=\left\{ \partial v_{1}^{\prime },\partial v_{2}^{\prime
},...\right\} .
\end{equation*}%
Since $\,b_{p-1}$ and $b_{p-1}^{\prime }$ are bases in the same subspace $%
B_{p-1}$, the transformation matrix $\left( c_{ij}\right) =\left( {b}%
_{p-1}^{\prime }/{b}_{p-1}\right) $ is well defined so that%
\begin{equation*}
\partial v_{i}^{\prime }=\sum_{j}c_{ij}\partial v_{j}.
\end{equation*}%
It follows that 
\begin{equation}
v_{i}^{\prime }=z_{i}+\sum_{j}c_{ij}v_{j}\ \ \text{for some }z_{i}\in Z_{p}.
\label{v'}
\end{equation}%
Since $Z_{p}=\limfunc{span}\left( b_{p},h_{p}\right) $, we obtain from (\ref%
{u'}) and (\ref{v'}) that 
\begin{equation*}
(b_{p}^{\prime },\widetilde{h}_{p}^{\prime },\widetilde{b}_{p}^{\prime
}\,/\,b_{p},\widetilde{h}_{p},\widetilde{b}_{p})=\left( 
\begin{array}{ccc}
(b_{p}^{\prime }/b_{p}) & \vdots & \vdots \\ 
0 & \func{id} & \vdots \\ 
0 & 0 & (b_{p-1}^{\prime }/b_{p-1})%
\end{array}%
\right)
\end{equation*}%
where the dots $\vdots $ denote the terms coming from $b_{i}$ and \thinspace 
$z_{i}$. Since this matrix is upper block-diagonal, we obtain (\ref{bhb}).

Consequently, we have 
\begin{eqnarray*}
\lbrack b_{p}^{\prime },\widetilde{h}_{p}^{\prime },\widetilde{b}%
_{p}^{\prime }\,/\,\omega _{p}] &=&[b_{p}^{\prime },\widetilde{h}%
_{p}^{\prime },\widetilde{b}_{p}^{\prime }\,/\,b_{p},\widetilde{h}_{p},%
\widetilde{b}_{p}][b_{p},\widetilde{h}_{p},\widetilde{b}_{p}\,/\,\omega _{p}]
\\
&=&[b_{p}^{\prime }/b_{p}][{b}_{p-1}^{\prime }/{b}_{p-1}][b_{p},\widetilde{h}%
_{p},\widetilde{b}_{p}\,/\,\omega _{p}].
\end{eqnarray*}%
Computing the sum in (\ref{tau}) we obtain 
\begin{eqnarray}
\sum_{p=0}^{N}(-1)^{p}\log [b_{p}^{\prime },\widetilde{h}_{p}^{\prime },%
\widetilde{b}_{p}^{\prime }\,/\,\omega _{p}] &=&\sum_{p=0}^{N}\left(
-1\right) ^{p}\log [b_{p},\widetilde{h}_{p},\widetilde{b}_{p}/\omega _{p}] 
\notag \\
&&+\sum_{p=0}^{N}\left( -1\right) ^{p}\log [b_{p}^{\prime
}/b_{p}]+\sum_{p=0}^{N}\left( -1\right) ^{p}\log [{b}_{p-1}^{\prime }/{b}%
_{p-1}].  \label{2}
\end{eqnarray}%
It remains to observe that the expression in (\ref{2}) vanishes because it
is equal to 
\begin{equation*}
\sum_{p=0}^{N}\left( -1\right) ^{p}\log [b_{p}^{\prime
}/b_{p}]+\sum_{q=-1}^{N-1}\left( -1\right) ^{q+1}\log [{b}_{q}^{\prime }/{b}%
_{q}]=\left( -1\right) ^{N}\log [b_{N}^{\prime }/b_{N}]=0.
\end{equation*}

$\left( b\right) $ Let $h_{p}=\left\{ \eta _{1},\eta _{2},...\right\} $ and $%
h_{p}^{\prime }=\left\{ \eta _{1}^{\prime },\eta _{2}^{\prime },...\right\} $
so that 
\begin{equation*}
\eta _{i}^{\prime }=\sum_{j}c_{ij}\eta _{j}
\end{equation*}%
where $\left( c_{ij}\right) =\left( h_{p}^{\prime }/h_{p}\right) .$ For the
representatives $u_{i}\in \widetilde{h}_{p}$ of $\eta _{i}$ and $%
u_{i}^{\prime }\in \widetilde{h}_{p}^{\prime }$ of $\eta _{i}^{\prime }$ in $%
Z_{p}$ we have then 
\begin{equation*}
u_{i}^{\prime }=\sum_{j}c_{ij}u_{ij}+b_{i}\ \ \text{for some }b_{i}\in B_{p}.
\end{equation*}%
It follows that 
\begin{equation*}
(b_{p},\widetilde{h}_{p}^{\prime },\widetilde{b}_{p}\,/\,b_{p},\widetilde{h}%
_{p},\widetilde{b}_{p}\,)=\left( 
\begin{array}{ccc}
\func{id} & \vdots & 0 \\ 
0 & (h_{p}^{\prime }/h_{p}) & 0 \\ 
0 & 0 & \func{id}%
\end{array}%
\right)
\end{equation*}%
where the dots $\vdots $ denote the terms coming from $b_{i}$. Hence, we
obtain%
\begin{eqnarray*}
\lbrack b_{p},\widetilde{h}_{p}^{\prime },\widetilde{b}_{p}\,/\,\omega
_{p}^{\prime }] &=&[b_{p},\widetilde{h}_{p}^{\prime },\widetilde{b}%
_{p}\,/\,b_{p},\widetilde{h}_{p},\widetilde{b}_{p}][b_{p},\widetilde{h}_{p},%
\widetilde{b}_{p}\,/\,\omega _{p}]\left[ \omega _{p}/\omega _{p}^{\prime }%
\right] \\
&=&[h_{p}^{\prime }/h_{p}][b_{p},\widetilde{h}_{p},\widetilde{b}%
_{p}\,/\,\omega _{p}]\left[ \omega _{p}/\omega _{p}^{\prime }\right] ,
\end{eqnarray*}%
whence (\ref{log+}) follows.
\end{proof}

Let us fix an inner product in each space $\Omega _{p}$ and denote by $\iota 
$ the \emph{inner product structure} in $\Omega ,$ that is, the collection
of all inner products for $p=0,...,N.$ Then we have the induced inner
product in the subspaces $B_{p}$, $Z_{p}$ and $\mathcal{H}_{p}.$ Using the
isomorphism $H_{p}\cong \mathcal{H}_{p}$ we transfer the inner product to $%
H_{p}.$ Hence, in this case we have a canonical choice of volume forms $%
\omega $ in $\Omega _{\ast }$ and $h$ in $H_{\ast }$ as we prefer
orthonormal bases $\omega _{p}$ in $\Omega _{p}$ and $h_{p}$ in $H_{p}$. In
fact, we can identify $h_{p}$ with an orthonormal basis in $\mathcal{H}_{p}$
and set $\widetilde{h}_{p}=h_{p}.$ With this choice of $\omega $ and $h$, we
define the R-torsion of $\left( \Omega ,\iota \right) $ by%
\begin{equation*}
\tau \left( \Omega ,\iota \right) =\tau \left( \Omega ,\omega ,h\right) .
\end{equation*}%
By (\ref{log+}) the right hand side does not depend on the choice of
orthonormal bases $\omega $ and $h$.

\begin{corollary}
Let $\iota $ and $\iota ^{\prime }$ be two inner product structures in $%
\Omega .$ Assume that there are positive reals $c_{p}$, $p=0,...,N$, such
that, for all $u,v\in \Omega _{p},$%
\begin{equation*}
\iota ^{\prime }\left( u,v\right) =c_{p}\iota \left( u,v\right) .
\end{equation*}%
Then 
\begin{equation}
\tau \left( \Omega ,\iota ^{\prime }\right) =\tau \left( \Omega ,\iota
\right) \prod_{p=0}^{N}c_{p}^{\frac{1}{2}\left( -1\right) ^{p}\left( \dim
\Omega _{p}-\dim H_{p}\right) }.  \label{taui'}
\end{equation}
\end{corollary}

In particular, if all $c_{p}=c$ are equal, then we obtain%
\begin{equation*}
\tau \left( \Omega ,\iota ^{\prime }\right) =\tau \left( \Omega ,\iota
\right)
\end{equation*}%
because 
\begin{equation*}
\sum_{p=0}^{N}\left( -1\right) ^{p}\dim \Omega _{p}=\sum_{p=0}^{N}\left(
-1\right) ^{p}\dim H_{p}=\chi \left( \Omega \right) ,
\end{equation*}%
where $\chi \left( \Omega \right) $ is the Euler characteristic of $\Omega .$

\begin{proof}
Since the notion of orthogonality is the same for $\iota $ and $\iota
^{\prime }$, the space $\mathcal{H}_{p}$ is also the same. If $\omega _{p}$
and $h_{p}$ are $\iota $-orthonormal bases in $\Omega _{p}$ and $\mathcal{H}%
_{p}$, respectively, then $\omega _{p}^{\prime }=\frac{1}{\sqrt{c_{p}}}%
\omega _{p}$ and $h_{p}^{\prime }=\frac{1}{\sqrt{c_{p}}}h_{p}$ are $\iota
^{\prime }$-orthonormal bases. Since 
\begin{equation*}
\left[ \omega _{p}/\omega _{p}^{\prime }\right] =c_{p}^{\dim \Omega _{p}/2}\
\ \ \text{and\ \ \ }\left[ h_{p}^{\prime }/h_{p}\right] =c_{p}^{-\dim
H_{p}/2},
\end{equation*}%
we obtain from (\ref{log+}) 
\begin{equation*}
\log \tau \left( \Omega ,\iota ^{\prime }\right) =\log \tau \left( \Omega
,\iota \right) +\frac{1}{2}\sum_{p=0}^{N}\left( -1\right) ^{p}\left( \dim
\Omega _{p}-\dim H_{p}\right) \log c_{p},
\end{equation*}%
whence (\ref{taui'}) follows.
\end{proof}

Let $\Omega $ be a chain complex that comes from a path complex $P.$ Then
the inner product in $\Omega _{p}$ can be taken from the ambient space $%
\mathcal{R}_{p}.$ The so obtained R-torsion $\tau \left( \Omega \left(
P\right) ,\iota \right) $ will also be denoted by $\tau \left( P,\iota
\right) .$

Let $\iota $ be the standard inner product (\ref{ii}) in $\mathcal{R}_{p}$
and $\iota ^{\prime }$ be the normalized inner product (\ref{iip}) in $%
\mathcal{R}_{p}.$ We set 
\begin{equation*}
\tau \left( P\right) =\tau \left( P,\iota \right) \ \ \text{and\ \ }\tau
^{\prime }\left( P\right) =\tau \left( P,\iota ^{\prime }\right) .
\end{equation*}%
In this case $c_{p}=\frac{1}{p!}$ and we obtain from (\ref{taui'}) 
\begin{equation}
\tau ^{\prime }\left( P\right) =\tau \left( P\right) \prod_{p=0}^{N}\left(
p!\right) ^{\frac{1}{2}\left( -1\right) ^{p+1}\left( \dim \Omega _{p}-\dim
H_{p}\right) }.  \label{logp!}
\end{equation}%
Denoting 
\begin{equation*}
r_{p}=\dim \Omega _{p}-\dim H_{p}
\end{equation*}%
and observing that 
\begin{equation*}
\sum_{p=0}^{N}\left( -1\right) ^{p}r_{p}=0
\end{equation*}%
and 
\begin{eqnarray*}
\sum_{p=2}^{N}\left( -1\right) ^{p+1}r_{p}\log \left( p!\right)
&=&\sum_{p=2}^{N}\left( -1\right) ^{p+1}r_{p}\sum_{k=2}^{p}\log k \\
&=&\sum_{k=2}^{N}\log k\sum_{p=k}^{N}\left( -1\right) ^{p+1}r_{p} \\
&=&\sum_{k=2}^{N}\log k\sum_{p=0}^{k-1}\left( -1\right) ^{p}r_{p} \\
&=&\sum_{k=2}^{N}\log k^{r_{0}-r_{1}+...+\left( -1\right) ^{k-1}r_{k-1}},
\end{eqnarray*}%
we obtain from (\ref{logp!})%
\begin{equation}
\tau ^{\prime }\left( P\right) =\tau \left( P\right) \left(
2^{r_{0}-r_{1}}\cdot 3^{r_{0}-r_{1}+r_{2}}\cdot
4^{r_{0}-r_{1}+r_{2}-r_{3}}\cdot ...\right) ^{1/2}.  \label{rk}
\end{equation}

\subsection{Examples}

\label{SecEx}Let us give some examples of computation of R-torsion by
definition.

\begin{example}
\label{ExI}Consider a \emph{line }digraph $G=\left( V,E\right) $ that
consists of $m$ vertices $V=\left\{ 0,1,...,m-1\right\} $ and $m-1$ arrows
having the form either $i\rightarrow i+1$ or $i+1\rightarrow i$, for $%
i=0,...,m-2.$ An example of a line digraph is shown in Fig. \ref{pic12}.%
\FRAME{ftbhFU}{3.8406in}{0.3606in}{0pt}{\Qcb{A line digraph with $m=5$}}{%
\Qlb{pic12}}{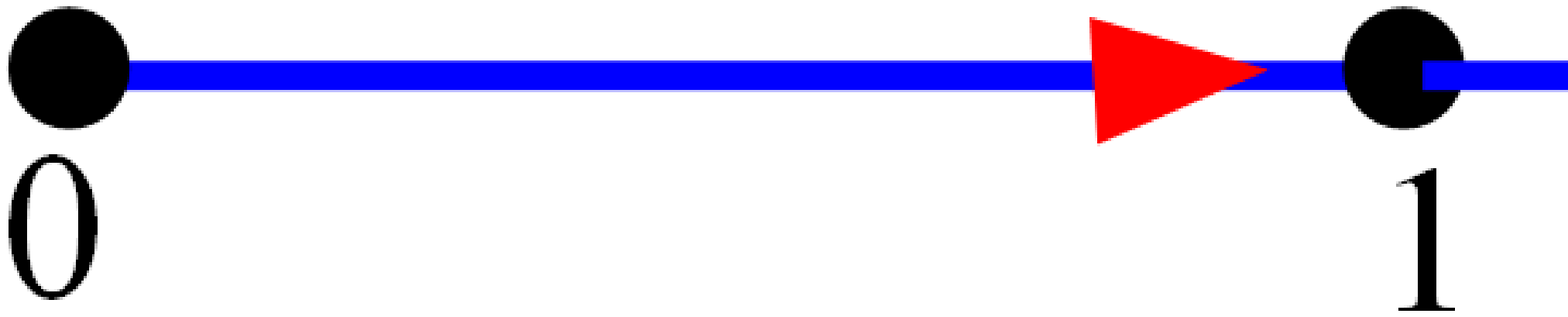}{\special{language "Scientific Word";type
"GRAPHIC";maintain-aspect-ratio TRUE;display "USEDEF";valid_file "F";width
3.8406in;height 0.3606in;depth 0pt;original-width 9.8009in;original-height
0.8934in;cropleft "0";croptop "1";cropright "1";cropbottom "0";filename
'pic12.eps';file-properties "XNPEU";}}

Denote%
\begin{equation}
\overline{e}_{i\left( i+1\right) }=\left\{ 
\begin{array}{ll}
e_{i\left( i+1\right) } & \text{if }i\rightarrow i+1 \\ 
e_{\left( i+1\right) i} & \text{if }i+1\rightarrow i%
\end{array}%
\right.  \label{ebar}
\end{equation}%
so that $\overline{e}_{i\left( i+1\right) }\in \Omega _{1}$, and set 
\begin{equation}
\sigma _{i}=\left\{ 
\begin{array}{ll}
1, & \text{if }i\rightarrow i+1 \\ 
-1, & \text{if }i+1\rightarrow i%
\end{array}%
\right.  \label{sigma}
\end{equation}%
so that 
\begin{equation*}
\partial \overline{e}_{i\left( i+1\right) }=\sigma _{i}\left(
e_{i+1}-e_{i}\right) .
\end{equation*}%
Choose the following $\iota $-orthonormal bases in $\Omega _{0}$ and $\Omega
_{1}$:%
\begin{equation*}
\omega _{0}=\left\{ e_{i}:i=0,...,m-1\right\}
\end{equation*}%
and%
\begin{equation*}
\omega _{1}=\left\{ \overline{e}_{i\left( i+1\right) }:i=0,...,m-2\right\} .
\end{equation*}%
Clearly, $\Omega _{p}=\left\{ 0\right\} $ for $p\geq 2.$ In particular, we
have $\chi \left( G\right) =1.$ Since $\dim H_{0}=1$ (as for any connected
graph) and $\dim H_{p}=0$ for $p\geq 2$, it follows that $\dim H_{1}=0.$

Since $B_{1}=\partial \Omega _{2}=\left\{ 0\right\} $, it follows that also $%
Z_{1}=\left\{ 0\right\} .$ We have $Z_{0}=\Omega _{0}$ and, hence, $\dim
B_{0}=m-1$. Choose in $B_{0}=\partial \Omega _{1}$ the basis 
\begin{equation*}
b_{0}=\left\{ \sigma _{i}\left( e_{i+1}-e_{i}\right) ,\ \ i=0,...,m-2\right\}
\end{equation*}%
and set, respectively, 
\begin{equation*}
\ \widetilde{b}_{1}=\left\{ \overline{e}_{i\left( i+1\right) },\
i=0,...,m-2\right\} .
\end{equation*}%
The orthogonal complement of $B_{0}$ in $Z_{0}$ is one-dimensional:%
\begin{equation*}
\mathcal{H}_{0}=\limfunc{span}\left\{ e_{0}+...+e_{m-1}\right\} ,
\end{equation*}%
so that 
\begin{equation*}
h_{0}=\left\{ \tfrac{1}{\sqrt{m}}\left( e_{0}+...+e_{m-1}\right) \right\} .
\end{equation*}%
We see that%
\begin{equation}
\lbrack b_{0},h_{0},\widetilde{b}_{0}\,/\,\omega _{0}]=\left\vert \det
\right\vert \left( 
\begin{array}{ccccccc}
-\sigma _{0} &  &  &  &  & 0 & \tfrac{1}{\sqrt{m}} \\ 
\sigma _{0} & -\sigma _{1} &  &  &  &  & \tfrac{1}{\sqrt{m}} \\ 
& \sigma _{1} & \ddots &  &  &  & \vdots \\ 
&  & \ddots & \ddots &  &  & \vdots \\ 
&  &  & \ddots & \ddots &  & \vdots \\ 
&  &  &  & \ddots & -\sigma _{m-2} & \vdots \\ 
0 &  &  &  &  & \sigma _{m-2} & \tfrac{1}{\sqrt{m}}%
\end{array}%
\right) =\sqrt{m},  \label{rootm}
\end{equation}%
because expanding the determinant in the last column, we obtain that is
equal to 
\begin{equation*}
\left( -1\right) ^{m+1}m\sigma _{0}...\sigma _{m-2}\frac{1}{\sqrt{m}}.
\end{equation*}%
Since $(b_{1},h_{1},\widetilde{b}_{1})\,=\omega _{1}$, it follows that 
\begin{equation*}
\tau \left( G\right) =\dprod\limits_{p=0}^{1}[b_{p},h_{p},\widetilde{b}%
_{p}\,/\,\omega _{p}]^{(-1)^{p}}=\sqrt{m}.
\end{equation*}%
For the normalized inner product $\iota ^{\prime }$ we have the same value $%
\tau ^{\prime }\left( G\right) =\sqrt{m}$ since $\Omega _{p}=\left\{
0\right\} $ for all $p\geq 2$.
\end{example}

\begin{example}
\label{Ex1}\label{Extriangle}Consider a digraph $G=\left( V,E\right) $ with
the vertex set $V=\left\{ 0,1,2\right\} $ and with the edge set $E=\left\{
01,12,02\right\} $ (Fig. \ref{pic7}). This digraph is called a \emph{triangle%
}. \FRAME{ftbhFU}{3.6092cm}{3.0482cm}{0pt}{\Qcb{A triangle digraph}}{\Qlb{%
pic7}}{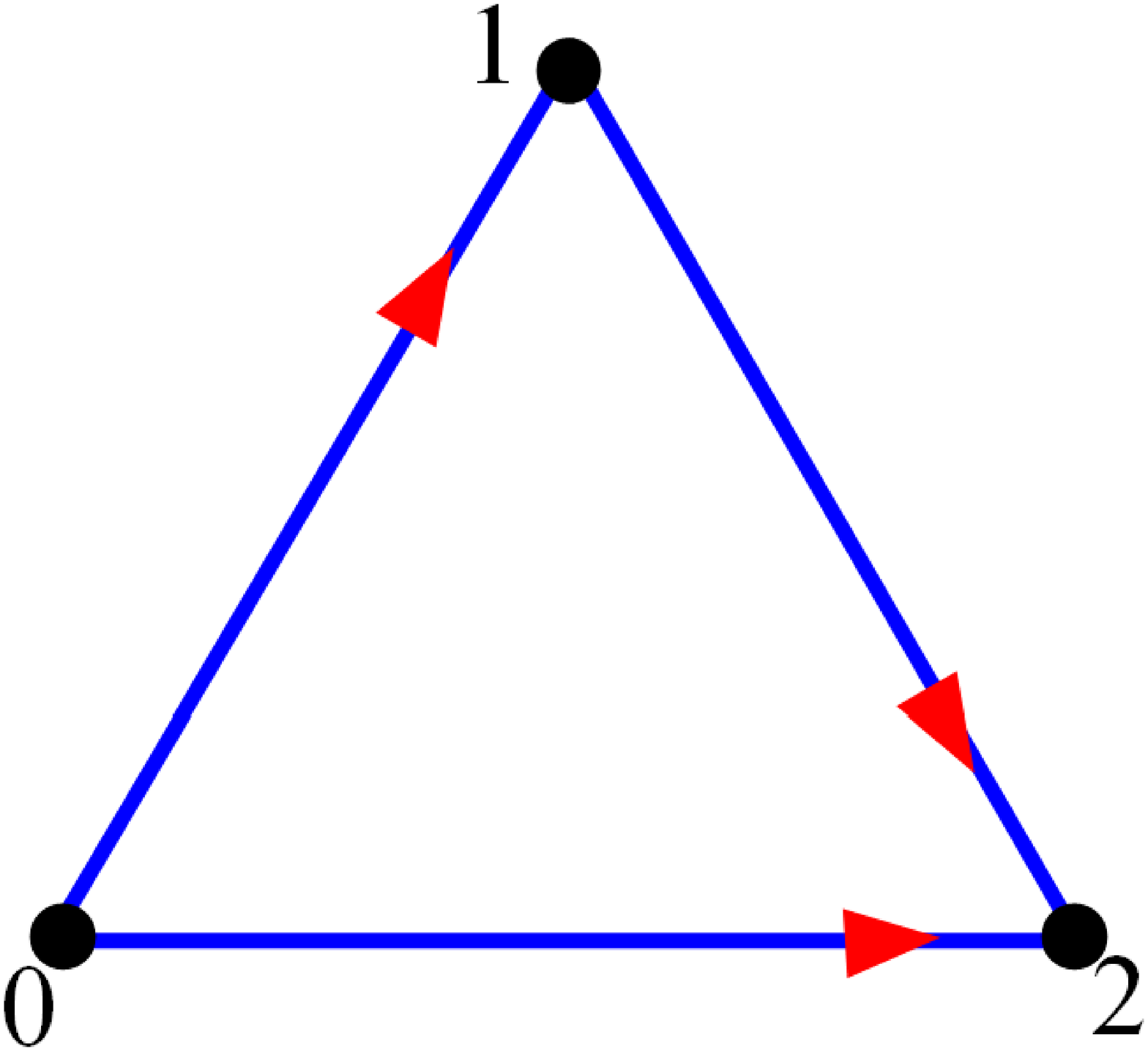}{\special{language "Scientific Word";type
"GRAPHIC";maintain-aspect-ratio TRUE;display "USEDEF";valid_file "F";width
3.6092cm;height 3.0482cm;depth 0pt;original-width 9.2258cm;original-height
7.7794cm;cropleft "0";croptop "1";cropright "1";cropbottom "0";filename
'pic7.eps';file-properties "XNPEU";}}

We have%
\begin{equation*}
\Omega _{0}=\limfunc{span}\left\{ e_{0},e_{1},e_{2}\right\} ,\ \ \Omega _{1}=%
\limfunc{span}\left\{ e_{01},e_{12},e_{02}\right\} ,\ \Omega _{2}=\left\{
e_{012}\right\}
\end{equation*}%
and $\Omega _{p}=\left\{ 0\right\} $ otherwise. Hence, 
\begin{equation*}
B_{0}=\partial \Omega _{1}=\limfunc{span}\left\{
e_{1}-e_{0},e_{2}-e_{1}\right\} ,\ \ \ B_{1}=\partial \Omega _{2}=\limfunc{%
span}\left\{ e_{01}-e_{02}+e_{12}\right\}
\end{equation*}%
and $B_{p}=\left\{ 0\right\} $ otherwise. Next, we have%
\begin{equation*}
Z_{0}=\limfunc{span}\left\{ e_{0},e_{1},e_{2}\right\} ,\ \ \ Z_{1}=\limfunc{%
span}\left\{ e_{01}-e_{02}+e_{12}\right\}
\end{equation*}%
and $Z_{p}=\left\{ 0\right\} $ otherwise. It follows that \thinspace $\dim
H_{0}=1$ and $\dim H_{p}=0$ otherwise.

We choose the following $\iota $-orthonormal bases in $\Omega _{p}$:%
\begin{equation*}
\omega _{0}=\left\{ e_{0},e_{1},e_{2}\right\} ,\ \ \ \omega _{1}=\left\{
e_{01},e_{12},e_{02}\right\} ,\ \ \omega _{2}=\left\{ e_{012}\right\} .
\end{equation*}%
Choose also 
\begin{eqnarray*}
b_{0} &=&\left\{ e_{1}-e_{0},e_{2}-e_{1}\right\} ,\ \ \ \widetilde{b}%
_{1}=\left\{ e_{01},e_{12}\right\} \\
b_{1} &=&\left\{ e_{01}-e_{02}+e_{12}\right\} ,\ \ \widetilde{b}_{2}=\left\{
e_{012}\right\} .
\end{eqnarray*}%
The orthogonal complement of $B_{0}$ in $Z_{0}$ is%
\begin{equation*}
\mathcal{H}_{0}=\limfunc{span}\left\{ e_{0}+e_{1}+e_{2}\right\} ,
\end{equation*}%
so that 
\begin{equation*}
h_{0}=\left\{ \tfrac{1}{\sqrt{3}}\left( e_{0}+e_{1}+e_{2}\right) \right\} .
\end{equation*}%
We see that%
\begin{equation*}
(b_{p},h_{p},\widetilde{b}_{p})=\left\{ 
\begin{array}{ll}
\{e_{1}-e_{0},e_{2}-e_{1},\frac{1}{\sqrt{3}}\left( e_{0}+e_{1}+e_{2}\right)
\}, & p=0 \\ 
\left\{ e_{01}-e_{02}+e_{12},e_{01},e_{12}\right\} , & p=1 \\ 
\left\{ e_{012}\right\} , & p=2.%
\end{array}%
\right.
\end{equation*}%
It follows that%
\begin{equation*}
\lbrack b_{0},h_{0},\widetilde{b}_{0}\,/\,\omega _{0}]=\left\vert \det
\right\vert \left( 
\begin{array}{ccc}
-1 & 0 & \frac{1}{\sqrt{3}} \\ 
1 & -1 & \frac{1}{\sqrt{3}} \\ 
0 & 1 & \frac{1}{\sqrt{3}}%
\end{array}%
\right) =\sqrt{3},
\end{equation*}%
\begin{equation*}
\lbrack b_{1},h_{1},\widetilde{b}_{1}\,/\,\omega _{1}]=\left\vert \det
\right\vert \left( 
\begin{array}{ccc}
1 & 1 & 0 \\ 
1 & 0 & 1 \\ 
-1 & 0 & 0%
\end{array}%
\right) =1
\end{equation*}%
and%
\begin{equation*}
\lbrack b_{2},h_{2},\widetilde{b}_{2}\,/\,\omega _{2}]=\left\vert \det
\right\vert \left( 1\right) =1.
\end{equation*}%
Hence, we obtain 
\begin{equation*}
\tau \left( G\right) =\dprod\limits_{p=0}^{2}[b_{p},h_{p},\widetilde{b}%
_{p}\,/\,\omega _{p}]^{(-1)^{p}}=\sqrt{3}.
\end{equation*}%
For the normalized inner product $\iota ^{\prime }$ we obtain from (\ref%
{taui'}) 
\begin{equation*}
\tau ^{\prime }\left( G\right) =\tau \left( G\right) \prod_{p=0}^{2}\left(
p!\right) ^{\frac{1}{2}\left( -1\right) ^{p+1}\left( \dim \Omega _{p}-\dim
H_{p}\right) }=\sqrt{3}2^{-\frac{1}{2}}=\sqrt{3/2}.
\end{equation*}
\end{example}

\begin{example}
\label{Ex2}\label{Exsquare}Consider a digraph $G=\left( V,E\right) $ with
the set of vertices $V=\left\{ 0,1,2,3\right\} $ and the set of edges $%
E=\left\{ 01,02,13,23\right\} $ (Fig. \ref{pic10}). This digraph is called a 
\emph{square}. \FRAME{ftbhFU}{3.2825cm}{2.9851cm}{0pt}{\Qcb{A square digraph}%
}{\Qlb{pic10}}{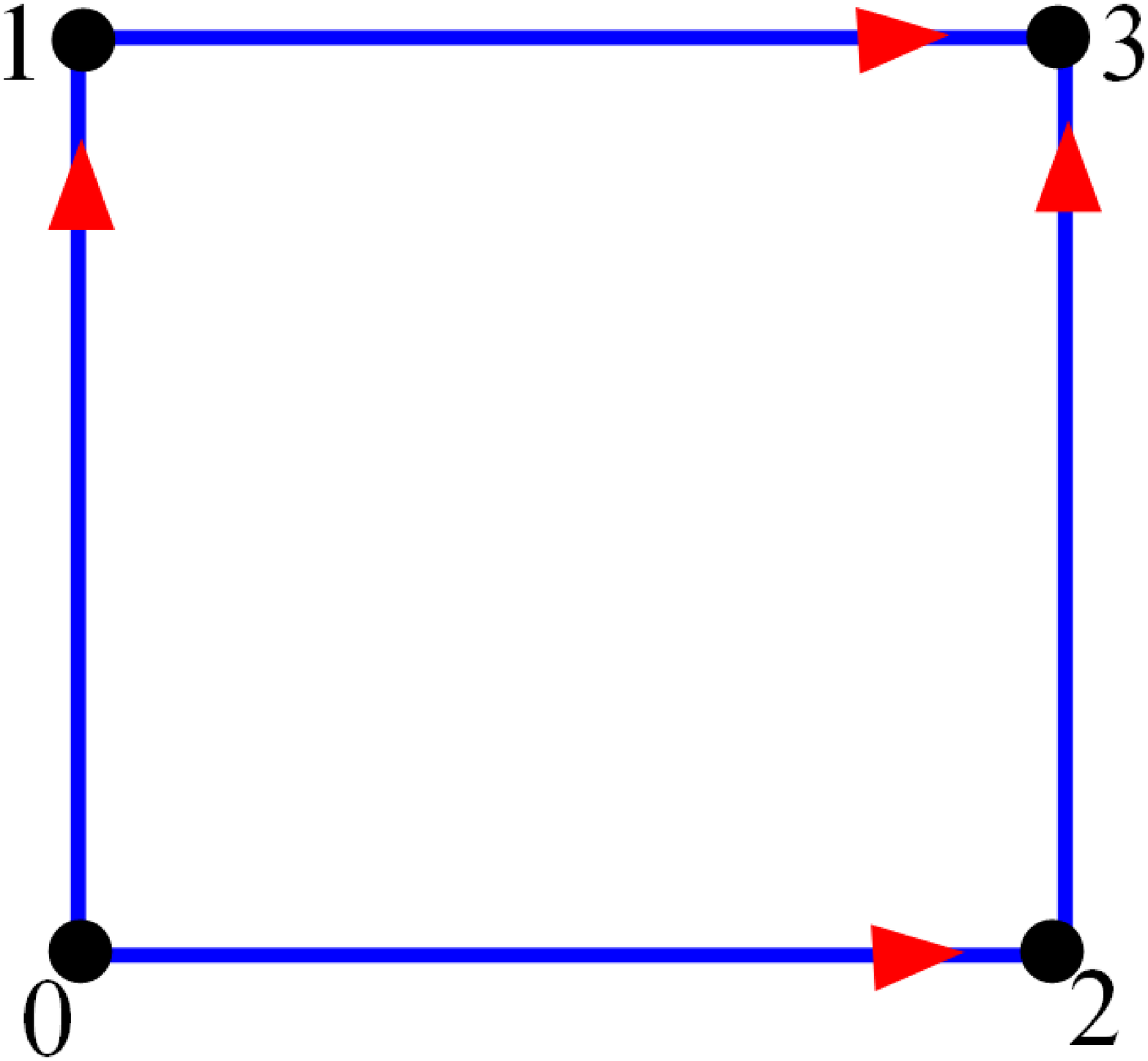}{\special{language "Scientific Word";type
"GRAPHIC";maintain-aspect-ratio TRUE;display "USEDEF";valid_file "F";width
3.2825cm;height 2.9851cm;depth 0pt;original-width 9.0816cm;original-height
8.2502cm;cropleft "0";croptop "1";cropright "1";cropbottom "0";filename
'pic10.eps';file-properties "XNPEU";}}

We have 
\begin{equation*}
\Omega _{0}=\limfunc{span}\left\{ e_{0},e_{1},e_{2},e_{3}\right\} ,\ \Omega
_{1}=\limfunc{span}\left\{ e_{01},e_{02},e_{13},e_{23}\right\} ,\ \ \Omega
_{2}=\limfunc{span}\left\{ e_{013}-e_{023}\right\}
\end{equation*}%
and $\Omega _{p}=\left\{ 0\right\} $ otherwise. Hence, 
\begin{eqnarray*}
B_{0} &=&\partial \Omega _{1}=\limfunc{span}\left\{
e_{1}-e_{0},e_{2}-e_{0},e_{3}-e_{1}\right\} \\
B_{1} &=&\partial \Omega _{2}=\limfunc{span}\left\{
e_{01}+e_{13}-e_{02}-e_{23}\right\}
\end{eqnarray*}%
and $B_{p}=\left\{ 0\right\} $ otherwise. Next we have%
\begin{equation*}
Z_{0}=\limfunc{span}\left\{ e_{0},e_{1},e_{2},e_{3}\right\} ,\ \ \ Z_{1}=%
\limfunc{span}\left\{ e_{01}+e_{13}-e_{02}-e_{23}\right\}
\end{equation*}%
and $Z_{p}=\left\{ 0\right\} $ otherwise. Consequently, $\,\dim H_{0}=1$ and 
$\dim H_{0}=0$ for $p\geq 1.$

We choose the following $\iota $-orthonormal bases in $\Omega _{p}$:%
\begin{equation*}
\omega _{0}=\left\{ e_{0},e_{1},e_{2},e_{3}\right\} ,\ \ \ \omega
_{1}=\left\{ e_{01},e_{02},e_{13},e_{23}\right\} ,\ \ \omega _{2}=\{\tfrac{1%
}{\sqrt{2}}\left( e_{013}-e_{023}\right) \}.
\end{equation*}%
Choose also 
\begin{eqnarray*}
b_{0} &=&\left\{ e_{1}-e_{0},e_{2}-e_{0},e_{3}-e_{1}\right\} ,\ \ \ 
\widetilde{b}_{1}=\left\{ e_{01},e_{02},e_{13}\right\} \\
b_{1} &=&\left\{ e_{01}-e_{02}+e_{13}-e_{23}\right\} ,\ \ \widetilde{b}%
_{2}=\left\{ e_{013}-e_{023}\right\} .
\end{eqnarray*}%
The orthogonal complement of $B_{0}$ in $Z_{0}$ is 
\begin{equation*}
\mathcal{H}_{0}=\limfunc{span}\left\{ e_{0}+e_{1}+e_{2}+e_{3}\right\}
\end{equation*}%
and we take 
\begin{equation*}
h_{0}=\left\{ \tfrac{1}{2}\left( e_{0}+e_{1}+e_{2}+e_{3}\right) \right\} .
\end{equation*}%
It follows that%
\begin{equation*}
(b_{p},h_{p},\widetilde{b}_{p})=\left\{ 
\begin{array}{ll}
\left\{ e_{1}-e_{0},e_{2}-e_{0},e_{3}-e_{1},\frac{1}{2}\left(
e_{0}+e_{1}+e_{2}+e_{3}\right) \right\} , & p=0 \\ 
\left\{ e_{01}-e_{02}+e_{13}-e_{23},e_{01},e_{02},e_{13}\right\} , & p=1 \\ 
\left\{ e_{013}-e_{023}\right\} , & p=2.%
\end{array}%
\right.
\end{equation*}%
Hence,\ 
\begin{equation*}
\lbrack b_{0},h_{0},\widetilde{b}_{0}\,/\,\omega _{0}]=\left\vert \det
\right\vert \left( 
\begin{array}{cccc}
-1 & -1 & 0 & \frac{1}{2} \\ 
1 & 0 & -1 & \frac{1}{2} \\ 
0 & 1 & 0 & \frac{1}{2} \\ 
0 & 0 & 1 & \frac{1}{2}%
\end{array}%
\right) =2,
\end{equation*}%
\begin{equation*}
\lbrack b_{1},h_{1},\widetilde{b}_{1}\,/\,\omega _{1}]=\left\vert \det
\right\vert \left( 
\begin{array}{cccc}
1 & 1 & 0 & 0 \\ 
-1 & 0 & -1 & 0 \\ 
1 & 1 & 0 & 1 \\ 
-1 & 0 & 0 & 0%
\end{array}%
\right) =1
\end{equation*}%
\begin{equation*}
\lbrack b_{2},h_{2},\widetilde{b}_{2}\,/\,\omega _{2}]=\left\vert \det
\right\vert (\sqrt{2})=\sqrt{2},
\end{equation*}%
and we obtain%
\begin{equation*}
\tau \left( G\right) =\dprod\limits_{p=0}^{2}[b_{p},h_{p},\widetilde{b}%
_{p}\,/\,\omega _{p}]^{(-1)^{p}}=2\sqrt{2}.
\end{equation*}%
For the normalized inner product $\iota ^{\prime }$ we obtain from (\ref%
{taui'}) 
\begin{equation*}
\tau ^{\prime }\left( G\right) =\tau \left( G\right) \prod_{p=0}^{2}\left(
p!\right) ^{\frac{1}{2}\left( -1\right) ^{p+1}\left( \dim \Omega _{p}-\dim
H_{p}\right) }=2\sqrt{2}2^{-\frac{1}{2}}=2.
\end{equation*}
\end{example}

Note that the triangle and square digraphs have the same homology groups and
are even homotopy equivalent (see \cite{Grigoryan-Lin-Muranov-Yau2014}) but
their torsions are different. Moreover, the torsion is not preserved by
covering mappings between digraphs, which are surjective mappings that
preserve arrows. For example, consider a mapping $\Phi :X\rightarrow Y$ of a
square $X$ on Fig. \ref{pic10} onto a line digraph $Y=\{0\rightarrow
1\rightarrow 2\}$ such that $\Phi \left( 0\right) =0$,\ $\Phi \left(
1\right) =\Phi \left( 2\right) =1$ and $\Phi \left( 3\right) =2,$ which is
obviously covering but $\tau \left( X\right) =2\sqrt{2}$ while $\tau \left(
Y\right) =\sqrt{3}.$

\begin{example}
\label{Excycle}We say that a digraph $G=\left( V,E\right) $ is \emph{cyclic}
if it is connected (as an undirected graph), every vertex had the degree $2,$
and there are no double arrows. For example, the triangle from Example \ref%
{Ex1} and the square from Example \ref{Ex2} are cyclic.

Here we assume that $G$ is neither triangle nor square. Some examples of
such digraphs are shown on Fig. \ref{pic16}. \FRAME{ftbhFU}{12.2086cm}{%
2.9311cm}{0pt}{\Qcb{Three cyclic digraphs with $3$, $4$ and $6$ vertices}}{%
\Qlb{pic16}}{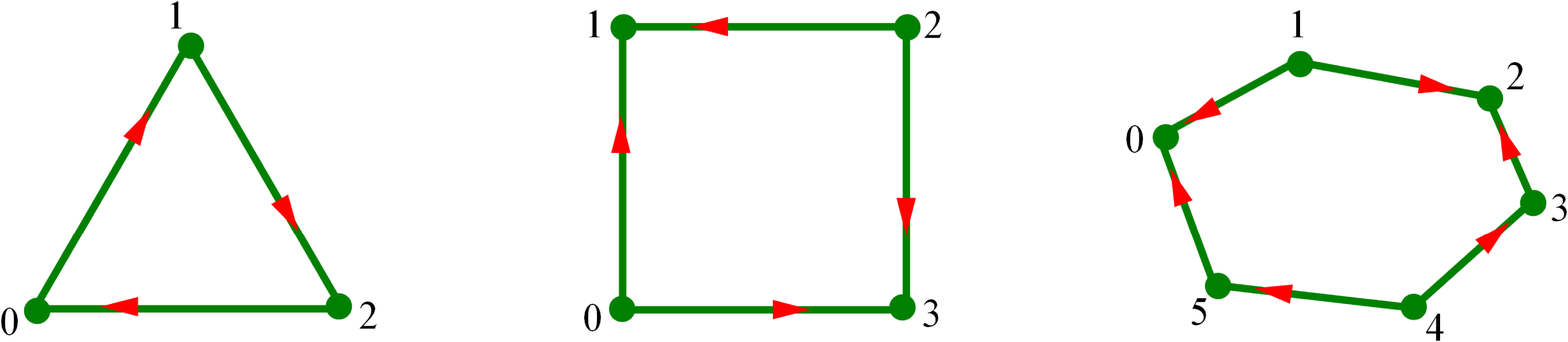}{\special{language "Scientific Word";type
"GRAPHIC";maintain-aspect-ratio TRUE;display "USEDEF";valid_file "F";width
12.2086cm;height 2.9311cm;depth 0pt;original-width 11.5778cm;original-height
8.9013cm;cropleft "0";croptop "1";cropright "1";cropbottom "0";filename
'pic16.eps';file-properties "XNPEU";}}

Note that a triangular digraph on Fig. \ref{pic16} is not a triangle in the
sense of Example \ref{Ex1} because of different orientation of the arrows,
and the quadrilateral digraph here is not a square for the same reason.

For a cyclic digraph that is neither triangle nor square, it is known that $%
\Omega _{p}\left( G\right) =\left\{ 0\right\} $ and $H_{p}\left( G\right)
=\left\{ 0\right\} $ for all $p\geq 2$, whereas 
\begin{equation*}
\dim H_{0}\left( G\right) =\dim H_{1}\left( G\right) =1
\end{equation*}%
and, hence, $\chi \left( G\right) =0$ (see \cite[Sect. 4.5]%
{Grigoryan-Lin-Muranov-Yau2013}). Assume that $G$ has $m$ vertices $%
0,1,...,m-1$ that we identify with residues \thinspace $\func{mod}m.$ The
numeration of vertices can be chosen so that all arrows have the form either 
$i\rightarrow i+1$ or $i+1\rightarrow i$, for $i=0,....,m-1.$

Let us use notations $\overline{e}_{i\left( i+1\right) }$ from (\ref{ebar})
and $\sigma _{i}$ from (\ref{sigma}) so that $\overline{e}_{i\left(
i+1\right) }\in \Omega _{1}$ and 
\begin{equation*}
\partial \overline{e}_{i\left( i+1\right) }=\sigma _{i}\left(
e_{i+1}-e_{i}\right) .
\end{equation*}%
Choose the following $\iota $-orthonormal bases in $\Omega _{0}$ and $\Omega
_{1}$:%
\begin{equation*}
\omega _{0}=\left\{ e_{i}:i=0,...,m-1\right\}
\end{equation*}%
and%
\begin{equation*}
\omega _{1}=\left\{ \overline{e}_{i\left( i+1\right) }:i=0,...,m-1\right\} .
\end{equation*}%
Observe that $Z_{0}=\Omega _{0}$ and 
\begin{equation*}
Z_{1}=\ker \partial |_{\Omega _{1}}=\limfunc{span}\left\{
\sum_{i=0}^{m-1}\sigma _{i}\overline{e}_{i\left( i+1\right) }\right\}
\end{equation*}%
because%
\begin{equation*}
\partial \left( \sum_{i}\alpha _{i}\overline{e}_{i\left( i+1\right) }\right)
=\sum_{i}\alpha _{i}\sigma _{i}\left( e_{i+1}-e_{i}\right) =\sum_{i}\left(
\alpha _{i-1}\sigma _{i-1}-\alpha _{i}\sigma _{i}\right) e_{i},
\end{equation*}%
which vanishes if $\alpha _{i}$ is proportional to $1/\sigma _{i}=\sigma
_{i}.$ Then $B_{0}=\partial \Omega _{1}$ has dimension $m-1$ and we choose 
\begin{equation*}
b_{0}=\left\{ \sigma _{i}\left( e_{i+1}-e_{i}\right) ,\ \ i=0,...,m-2\right\}
\end{equation*}%
and, respectively, 
\begin{equation*}
\ \widetilde{b}_{1}=\left\{ \overline{e}_{i\left( i+1\right) },\
i=0,...,m-2\right\} .
\end{equation*}%
The orthogonal complement of $B_{0}$ in $Z_{0}=\Omega _{0}$ is%
\begin{equation*}
\mathcal{H}_{0}=\limfunc{span}\left\{ e_{0}+...+e_{m-1}\right\} ,
\end{equation*}%
so that 
\begin{equation*}
h_{0}=\{\tfrac{1}{\sqrt{m}}\left( e_{0}+...+e_{m-1}\right) \}.
\end{equation*}%
Hence, as in (\ref{rootm}), we obtain%
\begin{equation*}
\lbrack b_{0},h_{0},\widetilde{b}_{0}\,/\,\omega _{0}]=\sqrt{m}.
\end{equation*}%
Next, we have $B_{1}=\partial \Omega _{2}=\left\{ 0\right\} $ whence $%
b_{1}=\emptyset $, $\mathcal{H}_{1}=Z_{1}$ and 
\begin{equation*}
h_{1}=\{\tfrac{1}{\sqrt{m}}\sum_{i=0}^{m-1}\sigma _{i}\overline{e}_{i\left(
i+1\right) }\}.
\end{equation*}%
\textbf{\ }We see that 
\begin{equation*}
\lbrack b_{1},h_{1},\widetilde{b}_{1}\,/\,\omega _{1}]=\left\vert \det
\right\vert \left( 
\begin{array}{cccccc}
\tfrac{1}{\sqrt{m}}\sigma _{0} & 1 &  &  &  & 0 \\ 
\tfrac{1}{\sqrt{m}}\sigma _{1} &  & 1 &  &  &  \\ 
\vdots &  &  & \ddots &  &  \\ 
\vdots &  &  &  & \ddots &  \\ 
\vdots &  &  &  &  & 1 \\ 
\frac{1}{\sqrt{m}}\sigma _{m-1} & 0 &  &  &  & 0%
\end{array}%
\right) =\frac{1}{\sqrt{m}}.
\end{equation*}%
It follows that%
\begin{equation*}
\tau \left( G\right) =\dprod\limits_{p=0}^{1}[b_{p},h_{p},\widetilde{b}%
_{p}\,/\,\omega _{p}]^{(-1)^{p}}=m
\end{equation*}%
and also $\tau ^{\prime }\left( G\right) =m.$
\end{example}

\subsection{Analytic torsion}

Let $\Omega $ be a chain complex as above and $\iota $ be an inner product
structure on $\Omega $. It is easy to check from (\ref{Dede}) that $\Delta
_{p}$ is a self-adjoint non-negative definite operator on $\Omega _{p}$.
Hence, its eigenvalues are non-negative reals, denote them by $\left\{
\lambda _{i}\right\} _{i=1}^{\dim \Omega _{p}}.$\ The zeta function $\zeta
_{p}\left( s\right) $ of $\Delta _{p}$ is defined by 
\begin{equation*}
\zeta _{p}(s)=\sum_{\lambda _{i}>0}\frac{1}{\lambda _{i}^{s}}.
\end{equation*}

\begin{definition}
\label{D:defi_of_ana_tor} The \emph{analytic torsion} $T(\Omega ,\iota )$ of
the chain complex $\Omega $ with an inner product structure $\iota $ is
defined by 
\begin{equation}
\log T(\Omega ,\iota )=\frac{1}{2}\sum_{p=0}^{N}(-1)^{p}\,p\,\zeta
_{p}^{\prime }(0).  \label{logT}
\end{equation}
\end{definition}

The next theorem is one of the main results of this paper.

\begin{theorem}
\label{T:main_thm}\label{L:log_tau(G)}We have\ \ 
\begin{equation*}
\tau (\Omega ,\iota )=T\left( \Omega ,\iota \right) .
\end{equation*}
\end{theorem}

This theorem was proved in \cite[Proposition 1.7]{Ray-Singer1971} for a
special case when the homology groups $H_{p}$ are trivial. We use a
modification of the argument of \cite{Ray-Singer1971} that works with
arbitrary homology groups.

\begin{proof}
Observe that%
\begin{equation*}
\zeta _{p}^{\prime }(s)=-\sum_{\lambda _{i}>0}\left( \log \lambda
_{i}\right) \lambda _{i}^{-s},
\end{equation*}%
whence%
\begin{equation}
\zeta _{p}^{\prime }\left( 0\right) =-\sum_{\lambda _{i}>0}\left( \log
\lambda _{i}\right) =-\log D_{p},  \label{xi'}
\end{equation}%
where 
\begin{equation*}
D_{p}:=\dprod\limits_{\lambda _{i}>0}\lambda _{i}
\end{equation*}%
is the determinant of $\Delta _{p}$ restricted on the direct sum of the
eigenspaces with positive eigenvalues. In the view of (\ref{logT}) and (\ref%
{xi'}), it suffices to prove that 
\begin{equation}
\log \tau (\Omega ,\iota )=-\frac{1}{2}\sum_{p=0}^{N}(-1)^{p}\,p\,\log D_{p}.
\label{logtau}
\end{equation}%
As before, we use notations $B_{p}=\func{im}\partial _{p+1}=\partial \Omega
_{p+1}$ and $Z_{p}=\ker \partial _{p}$ so that $H_{p}=Z_{p}/B_{p}$. Since
any element of $u\in B_{p}$ has the form $u=\partial v$ for some $v\in
\Omega _{p+1}$, we have%
\begin{equation}
\Delta _{p}u=\partial ^{\ast }\partial \partial v+\partial \partial ^{\ast
}u=\partial \left( \partial ^{\ast }u\right) \in B_{p}.  \label{uB}
\end{equation}%
Hence, $B_{p}$ is an invariant subspace of $\Delta _{p}$. Therefore, there
exists an orthonormal basis $b_{p}=\{b_{p}^{i}\}$ of $B_{p}$ that consists
of the eigenvectors of $\Delta _{p}$: 
\begin{equation*}
\Delta _{p}b_{p}^{i}=\beta _{p}^{i}b_{p}^{i},
\end{equation*}%
where $\beta _{p}^{i}$ are the corresponding eigenvalues. Since by (\ref{Hd}%
) $B_{p}$ is orthogonal to $\mathcal{H}_{p}$ and all the eigenvectors of $%
\Delta _{p}$ with eigenvalue $0$ belong to $\mathcal{H}_{p}$, we have $\beta
_{p}^{i}>0.$

By (\ref{uB}) we have $\Delta _{p}b_{p}^{i}=\partial \partial ^{\ast
}b_{p}^{i}$, whence%
\begin{equation}
\partial \partial ^{\ast }b_{p}^{i}=\beta _{p}^{i}b_{p}^{i}.  \label{eigen}
\end{equation}%
Set 
\begin{equation*}
\widetilde{b}_{p}^{i}:=\frac{1}{\beta _{p-1}^{i}}\partial ^{\ast
}b_{p-1}^{i}\in \Omega _{p}.
\end{equation*}%
We have by (\ref{eigen}) 
\begin{equation*}
\partial \widetilde{b}_{p}^{i}=\frac{1}{\beta _{p-1}^{i}}\partial \partial
^{\ast }b_{p-1}^{i}=\frac{1}{\beta _{p-1}^{i}}\cdot \beta
_{p-1}^{i}b_{p-1}^{i}=b_{p-1}^{i}
\end{equation*}%
so that the sequences $\widetilde{b}_{p}=\{\widetilde{b}_{p}^{i}\}$ and $%
b_{p-1}=\{b_{p-1}^{i}\}$ satisfy the identity (\ref{bp-1}) and, hence, can
be used in the definition of R-torsion. Since also 
\begin{equation*}
\partial ^{\ast }\widetilde{b}_{p}^{i}=\frac{1}{\beta _{p-1}^{i}}\partial
^{\ast }\partial ^{\ast }b_{p-1}^{i}=0,
\end{equation*}%
we obtain

\begin{equation*}
\Delta _{p}\widetilde{b}_{p}^{i}=\partial ^{\ast }\partial \widetilde{b}%
_{p}^{i}+\partial \partial ^{\ast }\widetilde{b}_{p}^{i}=\partial ^{\ast
}b_{p-1}^{i}+0=\beta _{p-1}^{i}\widetilde{b}_{p}^{i}.
\end{equation*}%
Hence, $\widetilde{b}_{p}^{i}$ are the eigenvectors of $\Delta _{p}$ with
eigenvalues $\beta _{p-1}^{i}.$ Moreover, the sequence $\,\{\widetilde{b}%
_{p}^{i}\}$ is orthogonal because by (\ref{eigen}) for $i\neq j$%
\begin{eqnarray*}
\langle \widetilde{b}_{p}^{i},\widetilde{b}_{p}^{j}\rangle &=&\frac{1}{\beta
_{p-1}^{i}\beta _{p-1}^{j}}\langle \partial ^{\ast }b_{p-1}^{i},\partial
^{\ast }b_{p-1}^{j}\rangle \\
&=&\frac{1}{\beta _{p-1}^{i}\beta _{p-1}^{j}}\langle \partial \partial
^{\ast }b_{p-1}^{i},b_{p-1}^{j}\rangle \\
&=&\frac{1}{\beta _{p-1}^{j}}\langle b_{p-1}^{i},b_{p-1}^{j}\rangle \\
&=&0.
\end{eqnarray*}%
In the case $i=j$ we obtain similarly 
\begin{equation*}
\Vert \widetilde{b}_{p}^{i}\Vert ^{2}=\frac{1}{\beta _{p-1}^{i}}\langle
b_{p-1}^{i},b_{p-1}^{i}\rangle =\frac{1}{\beta _{p-1}^{i}}.
\end{equation*}%
Note also that the vectors $b_{p}^{i}$ and $\widetilde{b}_{p}^{j}$ are
necessarily orthogonal since 
\begin{equation*}
\langle b_{p}^{i},\widetilde{b}_{p}^{j}\rangle =\frac{1}{\beta _{p-1}^{i}}%
\langle b_{p}^{i},\partial ^{\ast }b_{p-1}^{j}\rangle =\frac{1}{\beta
_{p-1}^{i}}\langle \partial b_{p}^{i},b_{p-1}^{j}\rangle =0.
\end{equation*}%
Let $h_{p}=\{h_{p}^{i}\}$ be an orthonormal basis of $\mathcal{H}_{p}.$ Then
the following sequence 
\begin{equation}
\omega _{p}=b_{p}\cup \{\sqrt{\beta _{p-1}^{i}}\,\widetilde{b}_{p}^{i}\}\cup
h_{p}  \label{om}
\end{equation}%
consists of the eigenvectors of $\Delta _{p}$ and is orthonormal. By
construction, this sequence is a basis in $\Omega _{p}$ (see Section \ref%
{SecR}). It follows that all the positive eigenvalues of $\Delta _{p}$ are%
\begin{equation*}
\left\{ \beta _{p-1}^{i}\right\} \cup \left\{ \beta _{p}^{i}\right\} ,
\end{equation*}%
whence 
\begin{equation*}
D_{p}=\prod_{i}\beta _{p-1}^{i}\prod_{i}\beta _{p}^{i}.
\end{equation*}%
Setting 
\begin{equation*}
L_{k}:=\log \prod_{i}\beta _{k}^{i}
\end{equation*}%
we obtain%
\begin{equation*}
\log D_{p}=L_{p-1}+L_{p}.
\end{equation*}%
Using that $L_{N}=0$, we obtain 
\begin{eqnarray}
L_{p-1} &=&\left( L_{p-1}+L_{p}\right) -\left( L_{p}+L_{p+1}\right) +... 
\notag \\
&=&\sum_{q=p}^{N}\left( -1\right) ^{q-p}\left( L_{q-1}+L_{q}\right)
=\sum_{q=p}^{N}(-1)^{q-p}\log D_{q}{.}  \label{L}
\end{eqnarray}%
It follows from (\ref{om}) that 
\begin{equation*}
\lbrack b_{p},\widetilde{b}_{p},h_{p}\,/\,\omega _{p}]=\prod_{i}\left( \beta
_{p-1}^{i}\right) ^{-1/2}=\left( L_{p-1}\right) ^{-1/2}.
\end{equation*}%
Using the definition of $\tau \left( \Omega ,\iota \right) $ and (\ref{L}),
we obtain 
\begin{eqnarray*}
\log \tau (\Omega ,\iota ) &=&\sum_{p=0}^{N}(-1)^{p}\log [b_{p},\widetilde{b}%
_{p},h_{p}\,/\,\omega _{p}] \\
&=&-\frac{1}{2}\sum_{p=1}^{N}(-1)^{p}L_{p-1} \\
&=&-\frac{1}{2}\sum_{p=1}^{N}(-1)^{p}\sum_{q=p}^{N}(-1)^{q-p}\log D_{q} \\
&=&-\frac{1}{2}\sum_{q=1}^{N}\sum_{p=1}^{q}(-1)^{q}\log D_{q} \\
&=&-\frac{1}{2}\sum_{q=1}^{N}(-1)^{q}q\log D_{q}{,}
\end{eqnarray*}%
which finishes the proof of (\ref{logtau})\textbf{.}
\end{proof}

\section{Cartesian product of path complexes}

\label{SecP}

\subsection{Product of paths}

Given two finite sets $X,Y$, consider their Cartesian product $Z=X\times Y.$
Let $z=z_{0}z_{1}...z_{r}$ be a regular elementary $r$-path on $Z$, where $%
z_{k}=\left( x_{k},y_{k}\right) $ with $x_{k}\in X$ and $y_{k}\in Y$.

\begin{definition}
We say that the path $z$ is \emph{step-like} if, for any $k=1,...,r$, either 
$x_{k-1}=x_{k}$ or $y_{k-1}=y_{k}$. In fact, exactly one of these conditions
holds as $z$ is regular.
\end{definition}

Any step-like path $z$ on $Z$ determines by projection regular elementary
paths $x$ on $X$ and $y$ on $Y$. More precisely, $x$ is obtained from $z$ by
taking the sequence of all $X$-components of the vertices of $z$ and then by
collapsing in it any subsequence of repeated vertices to one vertex. The
same rule applies to $y$. By construction, the projections $x$ and $y$ are
regular elementary paths on $X$ and $Y$, respectively. If the projections of 
$z=z_{0}...z_{r}$ are $x=x_{0}...x_{p}$ and $y=y_{0}...y_{q}$ then $p+q=r$
(cf. Fig. \ref{pic3022}(left)).

\FRAME{ftbhFU}{12.1275cm}{4.614cm}{0pt}{\Qcb{Left: a step-like path $z$ and
its projections $x$ and $y$. Right: a staircase $S\left( z\right) $ and its
elevation $L\left( z\right) $ (here $L\left( z\right) =30$)$.$}}{\Qlb{pic3022%
}}{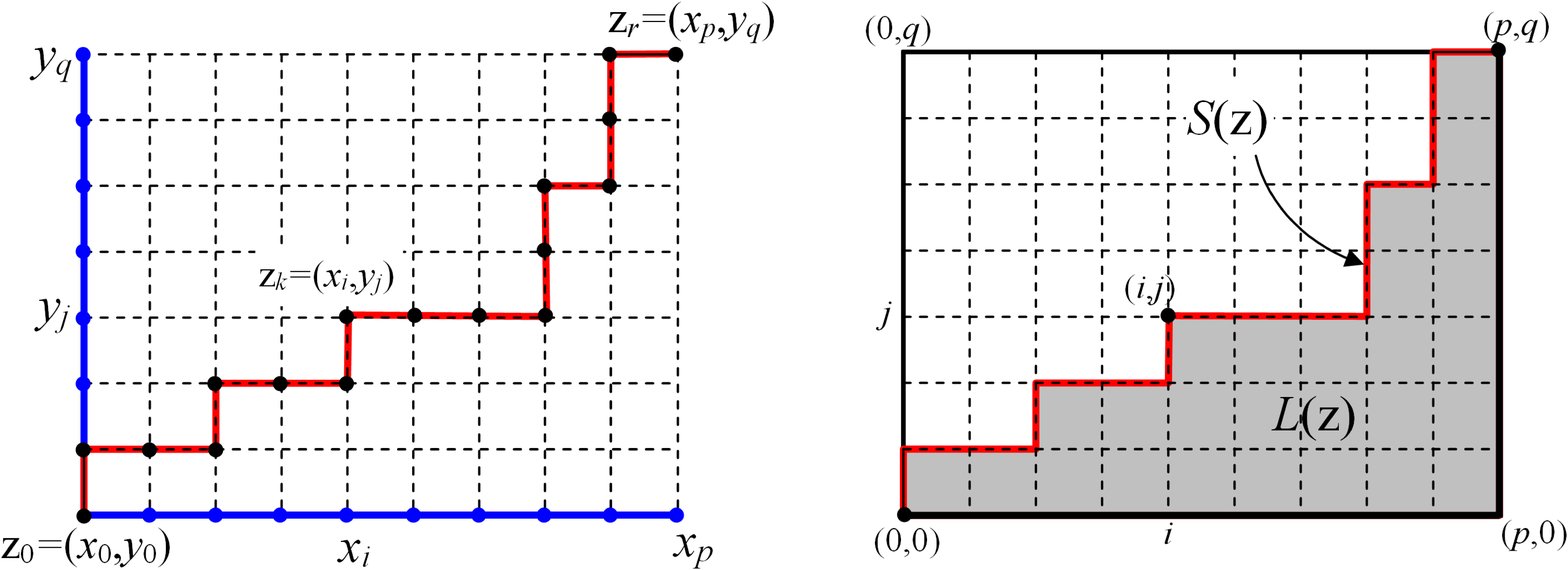}{\special{language "Scientific Word";type
"GRAPHIC";maintain-aspect-ratio TRUE;display "USEDEF";valid_file "F";width
12.1275cm;height 4.614cm;depth 0pt;original-width 31.1423cm;original-height
11.8009cm;cropleft "0";croptop "1";cropright "1";cropbottom "0";filename
'pic3022.eps';file-properties "XNPEU";}}

Every vertex $z_{k}=\left( x_{i},y_{j}\right) $ of a step-like path $z$ can
be represented as a point $\left( i,j\right) $ of $\mathbb{Z}^{2}$ so that
the whole path $z$ is represented by a \emph{staircase }$S\left( z\right) $
in $\mathbb{Z}^{2}$ connecting the points $\left( 0,0\right) $ and $\left(
p,q\right) $.

\begin{definition}
Define the \emph{elevation }$L\left( z\right) $ of the path $z$ as the
number of cells in $\mathbb{Z}_{+}^{2}$ below the staircase $S\left(
z\right) $ (the shaded area on Fig. \ref{pic3022}(right)).
\end{definition}

For given elementary regular $p$-path $x$ on $X$ and $q$-path $y$ on $Y$,
denote by $\Pi _{x,y}$ the set of all step-like paths $z$ on $Z$ whose
projections on $X$ and $Y$ are $x$ and $y$, respectively.

\begin{definition}
For regular elementary paths $e_{x}$ on $X$ and $e_{y}$ on $Y$ define their
cross product $e_{x}\times e_{y}$ as a path on $Z$ as follows: 
\begin{equation}
e_{x}\times e_{y}=\sum_{z\in \Pi _{x,y}}\left( -1\right) ^{L\left( z\right)
}e_{z}.  \label{uvz1}
\end{equation}%
Then extend by linearly the definition of $u\times v$ to all regular paths $%
u $ on $X$ and $v$ on $Y.$
\end{definition}

Clearly, if $u\in \mathcal{R}_{p}\left( X\right) $ and $v\in \mathcal{R}%
_{q}\left( Y\right) $ then $u\times v\in \mathcal{R}_{p+q}\left( Z\right) .$
Moreover, the cross product satisfies the product rule with respect to the
boundary operator $\partial $:%
\begin{equation}
\partial \left( u\times v\right) =\left( \partial u\right) \times v+\left(
-1\right) ^{p}u\times \left( \partial v\right)  \label{dusqv}
\end{equation}%
(see \cite[Prop. 6.3]{Grigoryan-Lin-Muranov-Yau2014}).

\subsection{Product of path complexes and digraphs}

\label{SecPP}

\begin{definition}
Given two finite sets $X$ and $Y$ with path complexes $P\left( X\right) $
and $P\left( Y\right) $ over $X$ and $Y$, respectively, define a path
complex $P\left( Z\right) $ over the set $Z=X\times Y$ as follows: the
elements of $P\left( Z\right) $ are step-like paths on $Z$ whose projections
on $X$ and $Y$ belong to $P\left( X\right) $ and $P\left( Y\right) $,
respectively. The path complex $P\left( Z\right) $ is called the \emph{%
Cartesian product} of the path complexes $P\left( X\right) $ and $P\left(
Y\right) $ and is denoted by $P\left( X\right) \Box P\left( Y\right) .$
\end{definition}

In short:\textbf{\ }a path $z$ on $Z$ is allowed if it is step-like and if
its projections on $X$ and $Y$ are allowed. In particular, if $x$ and $y$
are elementary allowed paths on $X$ and $Y$, respectively, then all the
paths $z\in \Pi _{x,y}$ are allowed on $Z$.

\begin{definition}
Let $X$ and $Y$ be digraphs. The Cartesian product $Z=X\Box Y$ of the
digraphs $X$ and $Y$ is defined as a digraph with the vertices $\left(
x,y\right) $ where $x\in X$ and $y\in Y$, and arrows $\left( x,y\right)
\rightarrow \left( x^{\prime },y^{\prime }\right) $ where either $%
x\rightarrow x^{\prime }$ and $y=y^{\prime }$ or $x=x^{\prime }$ and $%
y\rightarrow y^{\prime }$.
\end{definition}

For example, if $a\rightarrow a^{\prime }$ is an arrow in $X$ and $%
b\rightarrow b^{\prime }$ is an arrow in $Y$ then they induce the following
arrows in $Z$:%
\begin{equation*}
\begin{array}{ccc}
\overset{\left( a,b^{\prime }\right) }{\bullet } & \longrightarrow & \overset%
{\left( a^{\prime },b^{\prime }\right) }{\bullet } \\ 
\uparrow &  & \uparrow \\ 
\overset{\left( a,b\right) }{\bullet } & \longrightarrow & \overset{\left(
a^{\prime },b\right) }{\bullet }%
\end{array}%
\ 
\end{equation*}
Let $P\left( X\right) $ and $P\left( Y\right) $ be the path complexes in $X$
and $Y$, respectively, coming from the digraph structures. It is easy to see
that 
\begin{equation*}
P\left( X\Box Y\right) =P\left( X\right) \Box P\left( Y\right) ,
\end{equation*}%
that is, the Cartesian product of the path complexes is compatible with the
Cartesian product of digraphs. The reader who is interested only in digraphs
can always think of $X$ and $Y$ as digraphs and of $Z$ as their Cartesian
product.

For a general path complex $P\left( V\right) $ over a set $V$ we use the
short notations 
\begin{equation*}
\mathcal{A}_{p}\left( P\left( V\right) \right) \equiv \mathcal{A}_{p}\left(
V\right) \ \ \text{and\ \ }\Omega _{p}\left( P\left( V\right) \right) \equiv
\Omega _{p}\left( V\right) .
\end{equation*}%
It follows from (\ref{uvz1}) 
\begin{equation*}
u\in \mathcal{A}_{p}\left( X\right) \ \text{and\ }v\in \mathcal{A}_{q}\left(
Y\right) \ \ \Rightarrow \ \ u\times v\in \mathcal{A}_{p+q}\left( Z\right) .
\end{equation*}%
Moreover, (\ref{dusqv}) implies that 
\begin{equation*}
u\in \Omega _{p}\left( X\right) \ \text{and\ }v\in \Omega _{q}\left(
Y\right) \ \ \Rightarrow \ \ u\times v\in \Omega _{p+q}\left( Z\right)
\end{equation*}%
(see \cite[Prop. 6.5]{Grigoryan-Lin-Muranov-Yau2014}, \cite[Prop. 4.6]{GMY}%
). Furthermore, the following K\"{u}nneth formula is true: for any $r\geq 0$%
, 
\begin{equation}
\Omega _{r}\left( Z\right) =\bigoplus_{\left\{ p,q\geq 0:p+q=r\right\}
}\Omega _{p}\left( X\right) \otimes \Omega _{q}\left( Y\right) ,
\label{Hrpq}
\end{equation}%
where $\otimes $ denotes the tensor product of linear spaces, and $u\otimes
v $ for $u\in \Omega _{p}\left( X\right) $ and $v\in \Omega _{q}\left(
Y\right) $ is identified with the element $u\times v$ of $\Omega _{r}\left(
Z\right) $ (see \cite[Thm. 6.6]{Grigoryan-Lin-Muranov-Yau2014} and \cite[Thm
6.6]{Grigoryan-Lin-Muranov-Yau2020}).

\subsection{Operators $\partial ^{\ast }$ and $\Delta $ on products}

For the standard inner product $\iota $ defined by (\ref{ii}) on each of the
space $\mathcal{R}\left( X\right) $, $\mathcal{R}\left( Y\right) $ and $%
\mathcal{R}\left( Z\right) $ the following identity is known: if $\ u\in 
\mathcal{A}_{p}\left( X\right) $, $v\in \mathcal{A}_{q}\left( Y\right) $, $%
\varphi \in \mathcal{A}_{p^{\prime }}\left( X\right) $ and $\psi \in 
\mathcal{A}_{q^{\prime }}\left( Y\right) $, then%
\begin{equation*}
\langle u\times v,\varphi \times \psi \rangle _{\iota }=\tbinom{p+q}{p}%
\langle u,\varphi \rangle _{\iota }\langle v,\psi \rangle _{\iota }
\end{equation*}%
(see \cite[Lemma 4.13]{GMY}). This identity includes also the case when two
paths in the inner product have different length - in this case their inner
product is zero by definition. Hence, we have%
\begin{equation*}
\frac{1}{\left( p+q\right) !}\langle u\times v,\varphi \times \psi \rangle
_{\iota }=\frac{1}{p!}\langle u,\varphi \rangle _{\iota }\frac{1}{q!}\langle
v,\psi \rangle _{\iota }.
\end{equation*}%
In the case $p^{\prime }=p$ and $q^{\prime }=q$ we pass to the normalized
inner product $\iota ^{\prime }$ given by (\ref{iip}) and obtain 
\begin{equation}
\langle u\times v,\varphi \times \psi \rangle _{\iota ^{\prime }}=\langle
u,\varphi \rangle _{i^{\prime }}\langle v,\psi \rangle _{\iota ^{\prime }}~.
\label{uxvn}
\end{equation}%
This identity is true also if $p^{\prime }\neq p$ or $q^{\prime }\neq q$ as
in these cases the both sides vanish.

In the rest of this section we use the normalized inner product 
\begin{equation*}
\left\langle ,\right\rangle =\left\langle ,\right\rangle _{\iota ^{\prime }}
\end{equation*}%
unless otherwise specified. In particular, we define the adjoint operator $%
\partial ^{\ast }$ and the Hodge Laplacian with respect to the normalized
inner product and refer to them as \emph{normalized}.

\begin{lemma}
Let $u\in \Omega _{p}\left( X\right) $ and $v\in \Omega _{q}\left( Y\right)
. $ Then for the normalized adjoint operator we have%
\begin{equation}
\partial ^{\ast }\left( u\times v\right) =\partial ^{\ast }u\times v+\left(
-1\right) ^{p}u\times \left( \partial ^{\ast }v\right) .  \label{dprod}
\end{equation}
\end{lemma}

\begin{proof}
By definition, we have, for any $w\in \Omega _{p+q+1}\left( Z\right) $ 
\begin{equation*}
\langle \partial ^{\ast }\left( u\times v\right) ,w\rangle =\langle u\times
v,\partial w\rangle .
\end{equation*}%
Any $w\in \Omega _{\ast }\left( Z\right) $ admits a representation%
\begin{equation*}
w=\sum_{k}\varphi _{k}\times \psi _{k}
\end{equation*}%
where the sum is finite and%
\begin{equation*}
\varphi _{k}\in \Omega _{p_{k}}\left( X\right) \ \ \text{and }\psi _{k}\in
\Omega _{q_{k}}\left( Y\right)
\end{equation*}%
with $p_{k}+q_{k}=p+q+1$ (see \cite[Thm. 6.12]{Grigoryan-Lin-Muranov-Yau2014}%
, \cite[Theorem 5.1]{GMY}).

Then we have%
\begin{eqnarray*}
\langle \partial ^{\ast }\left( u\times v\right) ,w\rangle &=&\langle
u\times v,\sum \partial \left( \varphi _{k}\times \psi _{k}\right) \rangle \\
&=&\langle u\times v,\sum \left( \partial \varphi _{k}\times \psi
_{k}+\left( -1\right) ^{p_{k}}\varphi _{k}\times \partial \psi _{k}\right)
\rangle \\
&=&\sum \langle u\times v,\partial \varphi _{k}\times \psi _{k}\rangle
+\left( -1\right) ^{p_{k}}\langle u\times v,\varphi _{k}\times \partial \psi
_{k}\rangle \\
&=&\sum \langle u,\partial \varphi _{k}\rangle \langle v,\psi _{k}\rangle
+\left( -1\right) ^{p_{k}}\langle u,\varphi _{k}\rangle \langle v,\partial
\psi _{k}\rangle \\
&=&\sum \langle \partial ^{\ast }u,\varphi _{k}\rangle \langle v,\psi
_{k}\rangle +\left( -1\right) ^{p_{k}}\langle u,\varphi _{k}\rangle \langle
\partial ^{\ast }v,\psi _{k}\rangle .
\end{eqnarray*}%
Note that if $p_{k}\neq p$ then 
\begin{equation*}
\langle u,\varphi _{k}\rangle =0.
\end{equation*}%
Hence, we can replace $p_{k}$ everywhere by $p$ and obtain%
\begin{eqnarray*}
\langle \partial ^{\ast }\left( u\times v\right) ,w\rangle &=&\sum \langle
\partial ^{\ast }u,\varphi _{k}\rangle \langle v,\psi _{k}\rangle +\left(
-1\right) ^{p}\langle u,\varphi _{k}\rangle \langle \partial ^{\ast }v,\psi
_{k}\rangle \\
&=&\sum \langle \partial ^{\ast }u\times v,\varphi _{k}\times \psi
_{k}\rangle +\left( -1\right) ^{p}\langle u\times \partial ^{\ast }v,\varphi
_{k}\times \psi _{k}\rangle \\
&=&\langle \partial ^{\ast }u\times v+\left( -1\right) ^{p}u\times \partial
^{\ast }v,\sum \varphi _{k}\times \psi _{k}\rangle \\
&=&\langle \partial ^{\ast }u\times v+\left( -1\right) ^{p}u\times \partial
^{\ast }v,w\rangle ,
\end{eqnarray*}%
whence%
\begin{equation*}
\partial ^{\ast }\left( u\times v\right) =\partial ^{\ast }u\times v+\left(
-1\right) ^{p}u\times \partial ^{\ast }v.
\end{equation*}
\end{proof}

\begin{lemma}
For the normalized Hodge Laplacian we have%
\begin{equation}
\Delta \left( u\times v\right) =\left( \Delta u\right) \times v+u\times
\Delta v.  \label{dusqv2}
\end{equation}
\end{lemma}

\begin{proof}
Let $u\in \Omega _{p}\left( X\right) $ and $v\in \Omega _{q}\left( X\right) $%
. Then we have%
\begin{eqnarray*}
\partial \partial ^{\ast }\left( u\times v\right) &=&\partial \left(
\partial ^{\ast }u\times v+\left( -1\right) ^{p}u\times \partial ^{\ast
}v\right) \\
&=&\partial \left( \partial ^{\ast }u\times v\right) +\left( -1\right)
^{p}\partial \left( u\times \partial ^{\ast }v\right) \\
&=&\partial \partial ^{\ast }u\times v+\left( -1\right) ^{p+1}\partial
^{\ast }u\times \partial v \\
&&+\left( -1\right) ^{p}\left( \partial u\times \partial ^{\ast }v+\left(
-1\right) ^{p}u\times \partial \partial ^{\ast }v\right) \\
&=&\partial \partial ^{\ast }u\times v+\left( -1\right) ^{p+1}\partial
^{\ast }u\times \partial v+\left( -1\right) ^{p}\partial u\times \partial
^{\ast }v+u\times \partial \partial ^{\ast }v
\end{eqnarray*}%
and%
\begin{eqnarray*}
\partial ^{\ast }\partial \left( u\times v\right) &=&\partial ^{\ast }\left(
\partial u\times v+\left( -1\right) ^{p}u\times \partial v\right) \\
&=&\partial ^{\ast }\left( \partial u\times v\right) +\left( -1\right)
^{p}\partial ^{\ast }\left( u\times \partial v\right) \\
&=&\partial ^{\ast }\partial u\times v+\left( -1\right) ^{p-1}\partial
u\times \partial ^{\ast }v \\
&&+\left( -1\right) ^{p}\left( \partial ^{\ast }u\times \partial v+\left(
-1\right) ^{p}u\times \partial ^{\ast }\partial v\right) \\
&=&\partial ^{\ast }\partial u\times v+\left( -1\right) ^{p-1}\partial
u\times \partial ^{\ast }v+\left( -1\right) ^{p}\partial ^{\ast }u\times
\partial v+u\times \partial ^{\ast }\partial v.
\end{eqnarray*}%
Adding up the two identities and noticing that the terms $\partial ^{\ast
}u\times \partial v$ and $\partial u\times \partial ^{\ast }v$ cancel out,
we obtain%
\begin{equation*}
\Delta \left( u\times v\right) =\left( \Delta u\right) \times v+u\times
\left( \Delta v\right) .
\end{equation*}
\end{proof}

\subsection{Torsion of products}

Let $P\left( V\right) $ be a path complex on a set $V$ with the maximal
length $N$ . As before, let $\iota $ be the standard inner product structure
on $P\left( V\right) $ given by (\ref{ii}) and $\iota ^{\prime }$ be the
normalized inner product structure on $P\left( V\right) $ given by (\ref{iip}%
). Consider the corresponding standard and normalized torsions:%
\begin{equation*}
T\left( V\right) =T\left( \Omega \left( V\right) ,\iota \right) \ \ \ \text{%
and\ \ \ }T^{\prime }\left( V\right) =T\left( \Omega \left( V\right) ,\iota
^{\prime }\right) .
\end{equation*}%
In the same way we will use notation $\tau \left( V\right) $ and $\tau
^{\prime }\left( V\right) $ for R-torsions with respect to $\iota $ and $%
\iota ^{\prime }$, respectively. Since $T\left( V\right) =\tau \left(
V\right) $ and $T^{\prime }\left( V\right) =\tau ^{\prime }\left( V\right) $%
, the relation between $T\left( V\right) $ and $T^{\prime }\left( V\right) $
is given by (\ref{logp!}) and (\ref{rk}).

Although the main object of interest for us is the standard torsion $T\left(
V\right) $, in this section we make an essential use of $T^{\prime }\left(
V\right) $ as it behaves better with respect to the Cartesian product.

We need also the Euler characteristic of $\Omega \left( V\right) $:%
\begin{equation*}
\chi \left( V\right) =\chi \left( \Omega \left( V\right) \right)
=\sum_{p=0}^{N}\left( -1\right) ^{p}\dim \Omega _{p}\left( V\right)
=\sum_{p=0}^{N}\left( -1\right) ^{p}\dim H_{p}\left( V\right) .
\end{equation*}%
The next theorem is our main result about torsion on the product of path
complexes.

\begin{theorem}
\label{T:main_thm1} If $P\left( Z\right) =P\left( X\right) \Box P\left(
Z\right) $ then%
\begin{equation}
\log T^{\prime }(Z)=\chi (Y)\log T^{\prime }(X)+\chi (X)\log T^{\prime }(Y).
\label{logT2}
\end{equation}
\end{theorem}

Before the proof of Theorem \ref{T:main_thm1}, we need to do some
preparations. In the next lemmas we work with an arbitrary chain complex $%
\Omega $ with some inner product structure $\iota $. Let $\lambda $ be an
eigenvalue of the Hodge Laplacian $\Delta _{p}$ on some chain complex $%
\Omega .$ Consider the eigenspace of $\lambda $ and its subspaces: 
\begin{eqnarray*}
E_{p}(\lambda ) &=&\{\varphi \in \Omega _{p}:\Delta _{p}\varphi =\lambda
\varphi \}, \\
E_{p}^{\prime }(\lambda ) &=&\{\varphi \in E_{p}(\lambda ):\partial \varphi
=0\}, \\
E_{p}^{\prime \prime }(\lambda ) &=&\{\varphi \in E_{p}(\lambda ):\partial
^{\ast }\varphi =0\}.
\end{eqnarray*}%
In the case $\lambda =0$ these three spaces are identical by Lemma \ref%
{Lemharm}. In the case $\lambda \neq 0$ the situation is different.

\begin{lemma}
\label{LemE+}Assume that $\lambda \neq 0$. Then we have 
\begin{eqnarray}
E_{p}^{\prime }\left( \lambda \right) &=&\{\varphi \in \Omega _{p}:\partial
\partial ^{\ast }\varphi =\lambda \varphi \}  \label{E'} \\
E_{p}^{\prime \prime }(\lambda ) &=&\{\varphi \in \Omega _{p}:\partial
^{\ast }\partial \varphi =\lambda \varphi \}  \label{E''}
\end{eqnarray}%
and%
\begin{equation}
E_{p}(\lambda )=E_{p}^{\prime }(\lambda )\bigoplus E_{p}^{\prime \prime
}(\lambda ).  \label{E+}
\end{equation}
\end{lemma}

\begin{proof}
Let us first prove (\ref{E'}). If $\varphi \in E_{p}^{\prime }\left( \lambda
\right) $ then%
\begin{equation*}
\lambda \varphi =\Delta \varphi =\partial ^{\ast }\partial \varphi +\partial
\partial ^{\ast }\varphi =\partial \partial ^{\ast }\varphi .
\end{equation*}
Conversely, if $\partial \partial ^{\ast }\varphi =\lambda \varphi $ then 
\begin{equation*}
\partial \varphi =\frac{1}{\lambda }\partial \left( \partial \partial ^{\ast
}\varphi \right) =0
\end{equation*}
and, hence, $\Delta \varphi =\partial \partial ^{\ast }\varphi =\lambda
\varphi $ so that $\varphi \in E_{\lambda }^{\prime }.$ In the same way one
proves (\ref{E''}).

In order to verify (\ref{E+}), observe first that the space $E_{p}^{\prime
}\left( \lambda \right) $ and $E_{p}^{\prime \prime }\left( \lambda \right) $
are orthogonal because for any $\varphi \in E_{p}^{\prime }\left( \lambda
\right) $ and $\psi \in E_{p}^{\prime \prime }\left( \lambda \right) $ we
have%
\begin{equation*}
\left\langle \varphi ,\psi \right\rangle =\frac{1}{\lambda ^{2}}\left\langle
\partial \partial ^{\ast }\varphi ,\partial ^{\ast }\partial \psi
\right\rangle =\frac{1}{\lambda ^{2}}\left\langle \partial \partial \partial
^{\ast }\varphi ,\partial \psi \right\rangle =0.
\end{equation*}
For any $\varphi \in E_{p}\left( \lambda \right) $ we have%
\begin{equation*}
\left( \partial \partial ^{\ast }\right) ^{2}\varphi =\partial \partial
^{\ast }\left( \partial \partial ^{\ast }\varphi +\partial ^{\ast }\partial
\varphi \right) =\partial \partial ^{\ast }\Delta \varphi =\lambda \partial
\partial ^{\ast }\varphi ,
\end{equation*}%
which implies by (\ref{E'}) that $\partial \partial ^{\ast }\varphi \in
E_{p}^{\prime }.$ Similarly, we have $\partial ^{\ast }\partial \varphi \in
E_{p}^{\prime \prime }.$ Finally, for any $\varphi \in E_{p}\left( \lambda
\right) $ we have 
\begin{equation*}
\varphi =\frac{1}{\lambda }\Delta \varphi =\frac{1}{\lambda }\partial
\partial ^{\ast }\varphi +\frac{1}{\lambda }\partial ^{\ast }\partial
\varphi ,
\end{equation*}%
whence (\ref{E+}) follows.
\end{proof}

\begin{lemma}
\label{Lemiso}The operator $\lambda ^{-1/2}\partial $ is an isometry of $%
E_{p}^{\prime \prime }(\lambda )$ onto $E_{p-1}^{\prime }(\lambda )$ with
the inverse $\lambda ^{-1/2}\partial ^{\ast }$.
\end{lemma}

\begin{proof}
Let $\varphi \in E_{p}^{\prime \prime }(\lambda )$ so that $\partial ^{\ast
}\varphi =0$ and $\partial ^{\ast }\partial \varphi =\lambda \varphi $. For $%
\psi =\partial \varphi $ we have%
\begin{equation*}
\partial \partial ^{\ast }\psi =\partial \partial ^{\ast }\partial \varphi
=\lambda \partial \varphi =\lambda \psi
\end{equation*}
whence $\psi \in E_{p-1}^{\prime }\left( \lambda \right) .$ Hence, $\partial 
$ maps $E_{p}^{\prime \prime }\left( \lambda \right) $ into $E_{p-1}^{\prime
}\left( \lambda \right) .$ Let us verify that $\lambda ^{-1/2}\partial $ is
an isometry. For $\varphi \in E_{p}^{\prime \prime }\left( \lambda \right) $
and $\psi =\lambda ^{-1/2}\partial \varphi $ we have 
\begin{equation*}
\left\langle \psi ,\psi \right\rangle =\frac{1}{\lambda }\left\langle
\partial \varphi ,\partial \varphi \right\rangle =\frac{1}{\lambda }%
\left\langle \partial ^{\ast }\partial \varphi ,\varphi \right\rangle
=\left\langle \varphi ,\varphi \right\rangle .
\end{equation*}%
It remains to show that the mapping $\lambda ^{-1/2}\partial $ is onto and
has the inverse $\lambda ^{-1/2}\partial ^{\ast }.$ For any $\psi \in
E_{p-1}^{\prime }\left( \lambda \right) $ we have $\partial ^{\ast }\left(
\partial ^{\ast }\psi \right) =0$ and 
\begin{equation*}
\Delta _{p}(\partial ^{\ast }\psi )=\left( \partial ^{\ast }\partial
+\partial \partial ^{\ast }\right) \partial ^{\ast }\psi =\partial ^{\ast
}\partial \partial ^{\ast }\psi =\partial ^{\ast }(\lambda \psi )=\lambda
\partial ^{\ast }\psi ,
\end{equation*}%
which implies $\partial ^{\ast }\psi \in E_{p}^{\prime \prime }(\lambda ).$
Since by (\ref{E'}) $\partial \partial ^{\ast }\psi =\lambda \psi ,$ we
obtain 
\begin{equation*}
\lambda ^{-1/2}\partial (\lambda ^{-1/2}\partial ^{\ast }\psi )=\frac{1}{%
\lambda }\partial \partial ^{\ast }\psi =\psi .
\end{equation*}
and we conclude that $\lambda ^{-1/2}\partial $ and $\lambda ^{-1/2}\partial
^{\ast }$ are mutually inverse.
\end{proof}

Let $n_{p}(\lambda )$, $n_{p}^{\prime }(\lambda )$, $n_{p}^{\prime \prime
}(\lambda )$ be the dimensions of spaces $E_{p}(\lambda )$, $E_{p}^{\prime
}(\lambda )$, $E_{p}^{\prime \prime }(\lambda )$, respectively It follows
from Lemmas \ref{LemE+} and \ref{Lemiso} that%
\begin{eqnarray*}
n_{p-1}^{\prime }(\lambda ) &=&n_{p}^{\prime \prime }(\lambda ), \\
n_{p}(\lambda ) &=&n_{p}^{\prime }(\lambda )+n_{p}^{\prime \prime }(\lambda
).
\end{eqnarray*}%
As it follows from the definition of the Euler characteristic $\chi \left(
\Omega \right) $ and $H_{p}\cong \mathcal{H}_{p}$, we have%
\begin{equation*}
\chi \left( \Omega \right) =\sum_{p=0}^{N}\left( -1\right) ^{p}n_{p}\left(
0\right) .
\end{equation*}

\begin{lemma}
\label{LemNp}If $\lambda >0$ then 
\begin{equation*}
\sum\limits_{p=0}^{N}(-1)^{p}n_{p}(\lambda )=0.
\end{equation*}
\end{lemma}

\begin{proof}
We have%
\begin{eqnarray*}
\sum_{p=0}^{N}(-1)^{p}n_{p}(\lambda ) &=&\sum_{p=0}^{N}(-1)^{p}n_{p}^{\prime
}(\lambda )+\sum_{p=0}^{N}(-1)^{p}n_{p}^{\prime \prime }(\lambda ) \\
&=&\sum_{p=0}^{N}(-1)^{p}n_{p}^{\prime }(\lambda
)+\sum_{p=1}^{N}(-1)^{p}n_{p-1}^{\prime }(\lambda )+n_{0}^{\prime \prime
}(\lambda ) \\
&=&(-1)^{N}n_{N}^{\prime }(\lambda )+n_{0}^{\prime \prime }(\lambda ) \\
&=&0.
\end{eqnarray*}%
Here $n_{N}^{\prime }\left( \lambda \right) =0$ because for every vector $%
\varphi \in E_{N}^{\prime }\left( \lambda \right) $ we have $\partial ^{\ast
}\varphi =0$ and, hence, 
\begin{equation*}
\varphi =\frac{1}{\lambda }\partial \partial ^{\ast }\varphi =0,
\end{equation*}%
and $n_{0}^{\prime \prime }\left( \lambda \right) =0$ because for any $%
\varphi \in E_{0}^{\prime \prime }\left( \lambda \right) $ we have $\partial
\varphi =0$ and, hence, 
\begin{equation*}
\varphi =\frac{1}{\lambda }\partial ^{\ast }\partial \varphi =0.
\end{equation*}
\end{proof}

Now we can prove Theorem \ref{T:main_thm1}. The idea of proof is borrowed
from \cite[Thm. 2.5]{Ray-Singer1971}.

\begin{proof}[Proof of Theorem \protect\ref{T:main_thm1}]
The zeta function $\zeta _{p,X}\left( s\right) $ of $\Delta _{p}$ on $X$ can
be represented in the form 
\begin{equation*}
\zeta _{p,X}\left( s\right) =\sum_{\lambda >0}\lambda ^{-s}n_{p}\left(
\lambda ,X\right) ,
\end{equation*}%
where the sum is taken over all distinct positive eigenvalues $\lambda $ of $%
\Delta _{p}$ and $n_{p}\left( \lambda ,X\right) $ is the multiplicity of $%
\lambda $. Similar formulas hold for $\zeta _{q,Y}\left( s\right) $ and $%
\zeta _{r,Z}\left( s\right) .$

Let $u\in \Omega _{p}\left( X\right) $ be an eigenvector of $\Delta _{p}$
with eigenvalue $\lambda $ and $v\in \Omega _{q}\left( Y\right) $ be an
eigenvector of $\Delta _{q}$ with eigenvalue $\mu $. It follows from (\ref%
{dusqv2}) that $u\times v\in \Omega _{p+q}\left( Z\right) $ is an
eigenvector of $\Delta _{p+q}$ with eigenvalue $\lambda +\mu $. If $\left\{
u_{i}\right\} $ is an orthonormal basis in $\Omega _{p}\left( X\right) $
consisting of the eigenvectors of $\Delta _{p}$ and $\left\{ v_{j}\right\} $
is an orthonormal basis in $\Omega _{q}\left( Y\right) $ consisting of the
eigenvectors of $\Delta _{q}$ then the sequence $\left\{ u_{i}\times
v_{j}\right\} $ is orthonormal by (\ref{uxvn}) and, hence, forms a basis in $%
\Omega _{p}\left( X\right) \otimes \Omega _{q}\left( Y\right) .$

Let us fix $r\geq 0$ and recall that by the K\"{u}nneth formula (\ref{Hrpq}) 
$\Omega _{r}\left( Z\right) $ is a direct sum of the spaces $\Omega
_{p}\left( X\right) \otimes \Omega _{q}\left( Y\right) $ over all pairs $%
p,q\geq 0$ with $p+q=r.$ Hence, collecting all the bases in $\Omega
_{p}\left( X\right) \otimes \Omega _{q}\left( Y\right) $ of the form $%
\left\{ u_{i}\times v_{j}\right\} $ we obtain an orthonormal basis in $%
\Omega _{r}\left( Z\right) $. This basis consists of the eigenvectors of $%
\Delta _{r}$ in $\Omega _{r}\left( Z\right) .$ Hence, all the eigenvalues of 
$\Delta _{r}$ in $\Omega _{r}\left( Z\right) $ have the form $\lambda +\mu $
where $\lambda $ is an eigenvalue of $\Delta _{p}$ in $\Omega _{p}\left(
X\right) $, $\mu $ is an eigenvalue of $\Delta _{q}$ in $\Omega _{q}\left(
Y\right) $, and the multiplicity of $\lambda +\mu $ is $n_{p}\left( \lambda
,X\right) n_{q}\left( \mu ,Y\right) .$

Hence, we obtain 
\begin{equation}
\zeta _{r,Z}(s)=\sum_{\lambda +\mu >0}\sum_{p+q=r}(\lambda +\mu
)^{-s}n_{p}(\lambda ,X)n_{q}(\mu ,Y).  \label{zetan}
\end{equation}%
It follows that%
\begin{eqnarray}
&&\sum_{r\geq 0}(-1)^{r}r\zeta _{r,Z}(s)  \notag \\
&=&\sum_{\lambda +\mu >0}(\lambda +\mu )^{-s}\sum_{p\geq 0}\sum_{q\geq
0}(-1)^{p+q}(p+q)n_{p}(\lambda ,X)n_{q}(\mu ,Y)  \notag \\
&=&\sum_{\lambda +\mu >0}(\lambda +\mu )^{-s}\left( \sum_{p\geq
0}(-1)^{p}p\,n_{p}(\lambda ,X)\right) \left( \sum_{q\geq 0}(-1)^{q}n_{q}(\mu
,Y)\right)  \label{2a} \\
&+&\sum_{\lambda +\mu >0}(\lambda +\mu )^{-s}\left( \sum_{p\geq
0}(-1)^{p}n_{p}(\lambda ,X)\right) \left( \sum_{q\geq 0}(-1)^{q}q\,n_{q}(\mu
,Y)\right) .  \label{3a}
\end{eqnarray}%
By Lemma \ref{LemNp}, if $\lambda >0$ then 
\begin{equation*}
\sum_{p\geq 0}(-1)^{p}n_{p}(\lambda ,X)=0
\end{equation*}%
and if $\mu >0$ then%
\begin{equation*}
\sum_{q\geq 0}(-1)^{q}n_{q}(\mu ,Y)=0.
\end{equation*}%
Hence, in (\ref{2a})-(\ref{3a}) all the terms with $\lambda >0$ and $\mu >0$
vanish, and we obtain%
\begin{eqnarray*}
\sum_{r\geq 0}(-1)^{r}r\zeta _{r,Z}(s) &=&\sum_{\lambda >0,\mu =0}\lambda
^{-s}\left( \sum_{p\geq 0}(-1)^{p}p\,n_{p}(\lambda ,X)\right) \left(
\sum_{q\geq 0}(-1)^{q}n_{q}(0,Y)\right) \\
&+&\sum_{\mu >0,\lambda =0}\mu ^{-s}\left( \sum_{p\geq
0}(-1)^{p}n_{p}(0,X)\right) \left( \sum_{q\geq 0}(-1)^{q}q\,n_{q}(\mu
,Y)\right) \\
&=&\chi (Y)\sum_{p\geq 0}(-1)^{p}p\zeta _{p,X}(s)+\chi (X)\sum_{q\geq
0}(-1)^{q}q\zeta _{q,Y}(s).
\end{eqnarray*}%
Taking derivative of the both sides at $s=0$ and using the definition of
analytic torsion, we obtain 
\begin{equation}
\log T^{\prime }(Z)=\chi (Y)\log T^{\prime }(X)+\chi (X)\log T^{\prime }(Y).
\label{T2}
\end{equation}
\end{proof}

The K\"{u}nneth formula (\ref{Hrpq}) implies that 
\begin{equation*}
\chi \left( Z\right) =\chi \left( X\right) \chi \left( Y\right) ,
\end{equation*}%
which will be used in the next statement.

For any $n\geq 2$ define on the set $X^{\Box n}=\underset{n\ \mathrm{times}}{%
\underbrace{X\times ...\times X}}$ the following path complex $\,$%
\begin{equation*}
P(X^{\Box n})=\underset{n\ \mathrm{times}}{\underbrace{P\left( X\right) \Box
...\Box P\left( X\right) }}=P\left( X\right) ^{\Box n}.
\end{equation*}

\begin{corollary}
\label{Cor1n}We have%
\begin{equation}
\log T^{\prime }(X^{\Box n})=n\chi \left( X\right) ^{n-1}\log T^{\prime
}\left( X\right) .  \label{T2Xn}
\end{equation}
\end{corollary}

\begin{proof}
Denote $\log T^{\prime }\left( X^{\Box n}\right) =x_{n}$ and $\chi \left(
X\right) =a.$ Then $\chi \left( X^{n}\right) =a^{n}$, and we have by (\ref%
{T2}) 
\begin{equation*}
x_{n+1}=ax_{n}+a^{n}x_{1}.
\end{equation*}%
For $n=1$ (\ref{T2Xn}) is trivial. Assuming the induction hypothesis $%
x_{n}=na^{n-1}x_{1},$ we obtain 
\begin{equation*}
x_{n+1}=na^{n}x_{1}+a^{n}x_{1}=\left( n+1\right) a^{n}x_{1},
\end{equation*}%
which finishes the proof by induction.
\end{proof}

\begin{example}
\label{Excyclen}Let $G$ be a cyclic digraph with $m$ vertices from Example %
\ref{Excycle}. For $n\geq 2$ the product $G^{\Box n}$ can be regarded as an
analogue of a torus. Since $\chi \left( G\right) =0$, we obtain from (\ref%
{T2Xn}) that, for any $n\geq 2$,%
\begin{equation*}
T^{\prime }(G^{\Box n})=1.
\end{equation*}%
Recall for comparison that $T^{\prime }\left( G\right) =T\left( G\right) =m.$

Before we can compute $T(G^{\Box n})$, let us verify that%
\begin{equation}
\dim \Omega _{p}(G^{\Box n})=\tbinom{n}{p}m^{n}.  \label{Omcyclic}
\end{equation}%
Indeed, for $n=1$ this is true because 
\begin{equation*}
\dim \Omega _{0}\left( G\right) =\dim \Omega _{1}\left( G\right) =m.
\end{equation*}%
Assuming that (\ref{Omcyclic}) is true for some $n$, we obtain by the K\"{u}%
nneth formula (\ref{Hrpq})%
\begin{eqnarray*}
\dim \Omega _{r}(G^{\Box \left( n+1\right) }) &=&\sum_{p+q=r}\dim \Omega
_{p}(G^{\Box n})\dim \Omega _{q}\left( G\right) \\
&=&m\dim \Omega _{r}(G^{\Box n})+m\dim \Omega _{r-1}(G^{\Box n}) \\
&=&m\tbinom{n}{r}m^{n}+m\tbinom{n}{r-1}m^{n} \\
&=&m^{n+1}\tbinom{n+1}{r}.
\end{eqnarray*}%
In the same way, using that 
\begin{equation*}
\dim H_{0}\left( G\right) =\dim H_{1}\left( G\right) =1,
\end{equation*}%
we obtain that 
\begin{equation*}
\dim H_{p}(G^{\Box n})=\tbinom{n}{p}.
\end{equation*}%
Hence, by (\ref{logp!}) we obtain%
\begin{eqnarray*}
T(G^{\Box n}) &=&T^{\prime }(G^{\Box n})\prod_{p=0}^{n}\left( p!\right) ^{%
\frac{1}{2}\left( -1\right) ^{p}\left( \dim \Omega _{p}-\dim H_{p}\right) }
\\
&=&\prod_{p=2}^{n}\left( p!\right) ^{\frac{1}{2}\left( -1\right) ^{p}\binom{n%
}{p}\left( m^{n}-1\right) }.
\end{eqnarray*}%
In particular, we have $T(G^{\Box 2})=2^{\frac{1}{2}\left( m^{2}-1\right) }.$
\end{example}

\begin{example}
\label{ExIn}For the \emph{interval} $I=\left. ^{0}\bullet \rightarrow
\bullet ^{1}\right. $ we have by Example \ref{ExI} $T^{\prime }\left(
I\right) =\sqrt{2}$ and $\chi \left( I\right) =1.$ Consider the $n$%
-dimensional digraph \emph{cube} $I^{\Box n}.$ In the case $n=2$ it
coincides with the square from Example \ref{Ex2}, in the case $n=3$ this
digraph is shown on Fig. \ref{pic9}. \FRAME{ftbhFU}{4.258cm}{3.83cm}{0pt}{%
\Qcb{The cube $I^{\Box 3}$}}{\Qlb{pic9}}{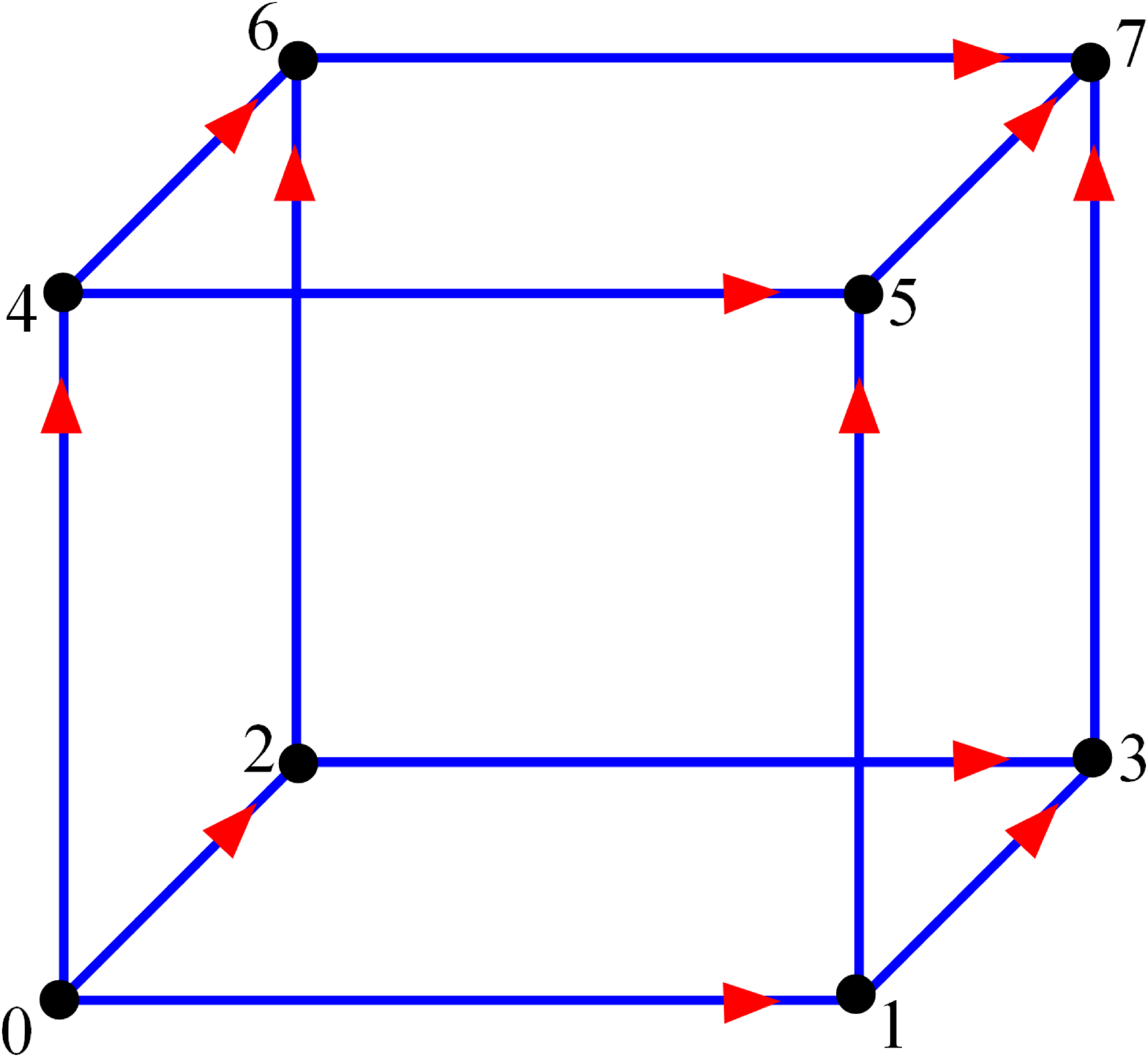}{\special{language
"Scientific Word";type "GRAPHIC";maintain-aspect-ratio TRUE;display
"USEDEF";valid_file "F";width 4.258cm;height 3.83cm;depth 0pt;original-width
16.5455cm;original-height 14.8716cm;cropleft "0";croptop "1";cropright
"1";cropbottom "0";filename 'pic9.eps';file-properties "XNPEU";}}

By (\ref{T2Xn}) we obtain 
\begin{equation*}
\log T^{\prime }(I^{\Box n})=n\log \sqrt{2}
\end{equation*}%
whence%
\begin{equation*}
T^{\prime }(I^{\Box n})=2^{n/2}.
\end{equation*}%
Let us compute the torsion $T(I^{\Box n})$ of the cube with respect to the
standard inner product $\iota $. For that let us first verify that%
\begin{equation}
\dim \Omega _{p}(I^{\Box n})=2^{n-p}\tbinom{n}{p},  \label{Omcube}
\end{equation}%
We have 
\begin{equation*}
\dim \Omega _{0}\left( I\right) =2,\ \ \dim \Omega _{1}\left( I\right) =1\ \ 
\text{and\ \ }\dim \Omega _{p}\left( I\right) =0\ \ \text{for }p\geq 2
\end{equation*}%
so that (\ref{Omcube}) holds for $n=1.$ For the inductive step from $n$ to $%
n+1$, observe that $I^{\Box \left( n+1\right) }=I^{\Box n}\Box I$. By the K%
\"{u}nneth formula (\ref{Hrpq}) we have%
\begin{eqnarray*}
\dim \Omega _{r}(I^{\Box \left( n+1\right) }) &=&\sum_{p+q=r}\dim \Omega
_{p}(I^{\Box n})\dim \Omega _{q}\left( I\right) \\
&=&2\dim \Omega _{r}(I^{\Box n})+\dim \Omega _{r-1}(I^{\Box n}) \\
&=&2^{n+1-r}\tbinom{n}{r}+2^{n-r+1}\tbinom{n}{r-1} \\
&=&2^{n+1-r}\tbinom{n+1}{r},
\end{eqnarray*}%
which finishes the proof of (\ref{Omcube}). In the same way one obtains that 
$\dim H_{0}(I^{\Box n})=1$ and $\dim H_{p}(I^{\Box n})=0$ for all $p\geq 1.$
Hence, by (\ref{logp!}) we obtain%
\begin{eqnarray*}
T(I^{\Box n}) &=&T^{\prime }(I^{\Box n})\prod_{p=0}^{n}\left( p!\right) ^{%
\frac{1}{2}\left( -1\right) ^{p}\left( \dim \Omega _{p}-\dim H_{p}\right) }
\\
&=&2^{n/2}\prod_{p=2}^{n}\left( p!\right) ^{\frac{1}{2}\left( -1\right)
^{p}2^{n-p}\binom{n}{p}}.
\end{eqnarray*}%
For example, we have 
\begin{equation*}
T(I^{\Box 2})=2\sqrt{2},\ \ \ \ T(I^{\Box 3})=\frac{16}{3}\sqrt{3},\ \ \ \
T(I^{\Box 4})=\frac{2048}{81}\sqrt{6},
\end{equation*}%
etc.
\end{example}

\begin{corollary}
\label{Cor1}If $P\left( Z\right) =P\left( X\right) \Box P\left( Z\right) $
then%
\begin{eqnarray}
&&\log T(Z)\underset{}{=}\chi (Y)\log T(X)+\chi (X)\log T(Y)  \notag \\
&&+\frac{1}{2}\sum_{p,q\geq 1}\left( -1\right) ^{p+q}\log \tbinom{p+q}{p}%
\left( \dim \Omega _{p}(X)\dim \Omega _{q}(Y)-\dim H_{p}(X)\dim
H_{q}(Y)\right) .  \label{T1}
\end{eqnarray}
\end{corollary}

\begin{proof}
We use the R-torsions $\tau $ and $\tau ^{\prime }$ defined with respect to
the inner products $\iota $ and $\iota ^{\prime }$, respectively. By
Theorems \ref{T:main_thm} and \ref{T:main_thm1} we have 
\begin{equation}
\log \tau ^{\prime }(Z)=\chi (Y)\log \tau ^{\prime }(X)+\chi (X)\log \tau
^{\prime }(Y).  \label{tau2}
\end{equation}%
By (\ref{logp!}) we have%
\begin{equation*}
\log \tau ^{\prime }(X)=\log \tau (X)-\frac{1}{2}\sum_{p}\left( -1\right)
^{p}\dim \Omega _{p}\left( X\right) \log \left( p!\right) +\frac{1}{2}%
\sum_{p}\left( -1\right) ^{p}\dim H_{p}\left( X\right) \log \left( p!\right)
,
\end{equation*}%
where summation is taken over all $p\geq 0$. Similar identities hold for $Y$
and $Z$. Substituting into (\ref{tau2}), we obtain 
\begin{eqnarray}
&&\log \tau (Z)-\chi (Y)\log \tau (X)-\chi (X)\log \tau (Y)  \notag \\
&=&\frac{1}{2}\sum_{r}\left( -1\right) ^{r}\dim \Omega _{r}\left( Z\right)
\log \left( r!\right) -\frac{1}{2}\sum_{r}\left( -1\right) ^{r}\dim
H_{r}\left( Z\right) \log \left( r!\right)  \notag \\
&&-\frac{1}{2}\chi (Y)\sum_{p}\left( -1\right) ^{p}\dim \Omega _{p}\left(
X\right) \log \left( p!\right) +\frac{1}{2}\chi (Y)\sum_{p}\left( -1\right)
^{p}\dim H_{p}\left( X\right) \log \left( p!\right)  \notag \\
&&-\frac{1}{2}\chi (X)\sum_{q}\left( -1\right) ^{q}\dim \Omega _{q}\left(
Y\right) \log \left( q!\right) +\frac{1}{2}\chi (X)\sum_{q}\left( -1\right)
^{q}\dim H_{q}\left( Y\right) \log \left( q!\right)  \label{3r}
\end{eqnarray}%
Denote for simplicity 
\begin{equation*}
x_{p}=\dim \Omega _{p}\left( X\right) ,\ \ \ y_{q}=\dim \Omega _{q}\left(
Y\right) ,\ \ \ z_{r}=\dim \Omega _{r}\left( Z\right) .
\end{equation*}%
By the K\"{u}nneth formula we have%
\begin{equation*}
z_{r}=\sum_{p+q=r}x_{p}y_{q}.
\end{equation*}%
It follows that%
\begin{eqnarray*}
&&\sum_{r}\left( -1\right) ^{r}z_{r}\log \left( r!\right) -\chi \left(
Y\right) \sum_{p}\left( -1\right) ^{p}x_{p}\log \left( p!\right) -\chi
\left( X\right) \sum_{q}\left( -1\right) ^{q}y_{p}\log \left( q!\right) \\
&=&\sum_{r}\left( -1\right) ^{r}\sum_{p+q=r}x_{p}y_{q}\log \left( r!\right)
-\sum_{q}\left( -1\right) ^{q}y_{q}\sum_{p}\left( -1\right) ^{p}x_{p}\log
\left( p!\right) -\sum_{p}\left( -1\right) ^{p}x_{p}\sum_{q\geq 0}\left(
-1\right) ^{q}y_{p}\log \left( q!\right) \\
&=&\sum_{p,q}\left( -1\right) ^{p+q}x_{p}y_{q}\log \left( \left( p+q\right)
!\right) -\sum_{p,q}\left( -1\right) ^{p+q}x_{p}y_{q}\log \left( p!\right)
-\sum_{p,q}\left( -1\right) ^{p+q}x_{p}y_{q}\log \left( q!\right) \\
&=&\sum_{p,q}\left( -1\right) ^{p+q}x_{p}y_{q}\log \tbinom{p+q}{p}.
\end{eqnarray*}%
Note that the summation here can be restricted to $p,q\geq 1$ since
otherwise $\log \binom{p+q}{p}=0.$ A similar formula takes place for $\dim
H_{p}$ instead of $\dim \Omega _{p}.$ Substituting into (\ref{3r}) we obtain
(\ref{T1}).
\end{proof}

\begin{example}
\label{ExPrism}Let us compute the torsions of the digraph $Z=I\Box Y$ where $%
I$ is the interval from Example \ref{ExIn} and $Y$ is the triangle from
Example \ref{Ex1} (see Fig. \ref{pic8}). \FRAME{ftbhFU}{4.044cm}{4.0237cm}{%
0pt}{\Qcb{A prism digraph $I\Box Y$}}{\Qlb{pic8}}{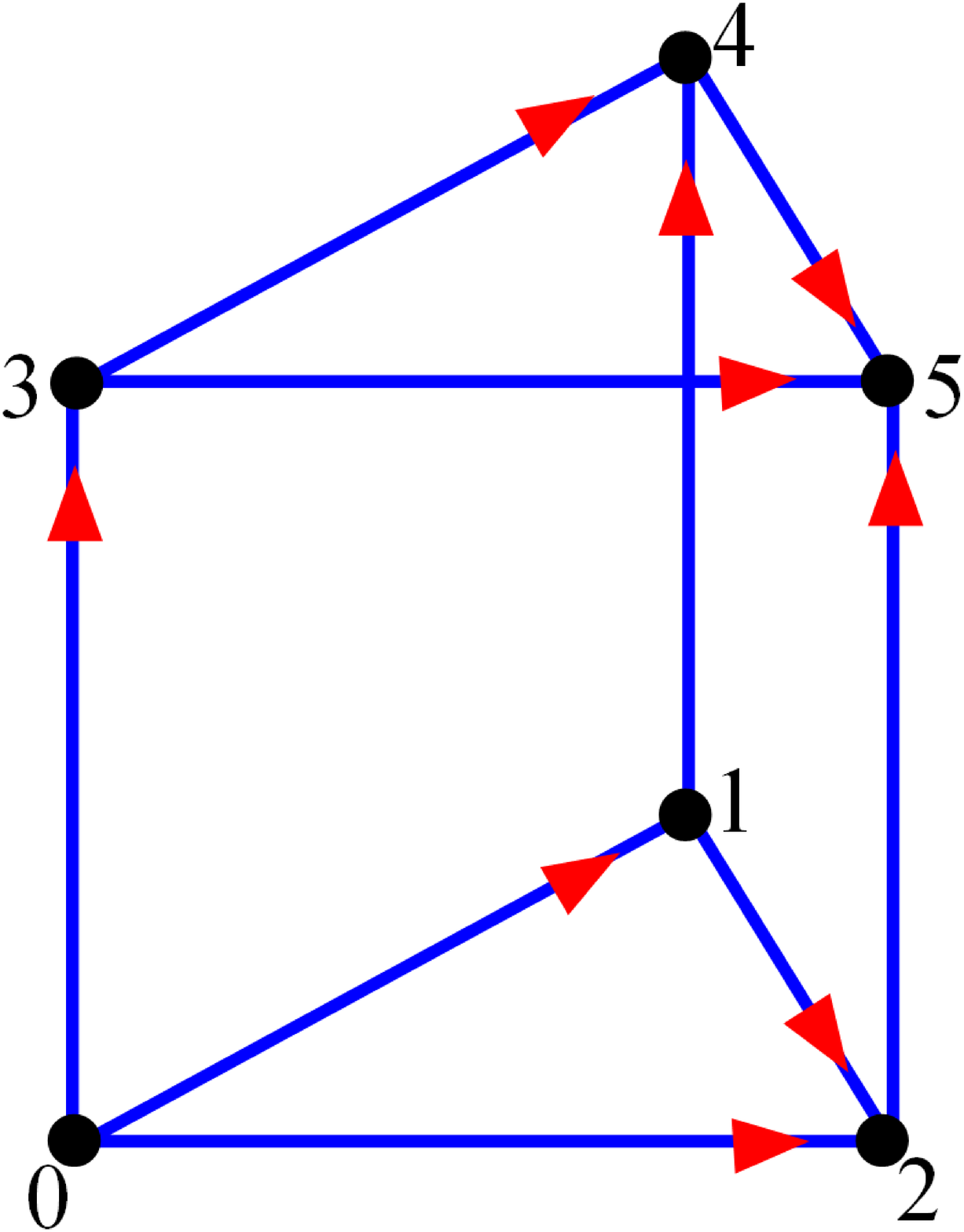}{\special%
{language "Scientific Word";type "GRAPHIC";maintain-aspect-ratio
TRUE;display "USEDEF";valid_file "F";width 4.044cm;height 4.0237cm;depth
0pt;original-width 13.423cm;original-height 13.3509cm;cropleft "0";croptop
"1";cropright "1";cropbottom "0";filename 'pic8.eps';file-properties
"XNPEU";}}

By Examples \ref{ExI} and \ref{Ex1}, we have $\chi \left( I\right) =\chi
\left( Y\right) =1$ and 
\begin{equation*}
T^{\prime }\left( I\right) =\sqrt{2},\ \ \ \ T^{\prime }\left( Y\right) =%
\sqrt{3/2}.
\end{equation*}%
Hence, we obtain by (\ref{logT2}) 
\begin{equation*}
T^{\prime }\left( Z\right) =T^{\prime }\left( I\right) ^{\chi \left(
Y\right) }T^{\prime }\left( Y\right) ^{\chi \left( I\right) }=\sqrt{2}\sqrt{%
3/2}=\sqrt{3}.
\end{equation*}%
Since $H_{p}\left( I\right) $ and $H_{p}\left( Y\right) $ are non-trivial
only for $p=0,$ we obtain by (\ref{T1})%
\begin{eqnarray*}
\log T(Z) &=&\chi (Y)\log T(I)+\chi (I)\log T(Y) \\
&&+\frac{1}{2}\sum_{p=1}^{1}\sum_{q=1}^{2}\left( -1\right) ^{p+q}\log 
\tbinom{p+q}{p}\dim \Omega _{p}\left( I\right) \dim \Omega _{q}\left(
Y\right) \\
&=&\log \sqrt{2}+\log \sqrt{3}+\frac{1}{2}\log \binom{2}{1}\cdot 1\cdot 3-%
\frac{1}{2}\log \binom{3}{1}\cdot 1\cdot 1 \\
&=&\log \left( \sqrt{2}\sqrt{3}2^{3/2}3^{-1/2}\right) =\log 4,
\end{eqnarray*}%
so that%
\begin{equation*}
T\left( Z\right) =4.
\end{equation*}
\end{example}

\section{Join of path complexes}

\label{SecJ}

\subsection{Augmented chain complex}

Let $P$ be a path complex over a set $V$ as in Section \ref{SecPC}. In that
section we have constructed a chain complex $\mathcal{R}=\left\{ \mathcal{R}%
_{p}\right\} _{p\geq 0}$ with the boundary operator $\partial $, and the
space $\mathcal{R}_{-1}$ was defined as $\left\{ 0\right\} .$ In this
section we change the definition of $\mathcal{R}_{-1}$ as follows. For any
elementary $0$-paths $e_{i}$ redefine $\partial $ by 
\begin{equation*}
\partial e_{i}=e
\end{equation*}%
where $e$ is an empty path that by definition has the length $-1.$ Set $%
\mathcal{R}_{-1}=\limfunc{span}_{\mathbb{R}}\left\{ e\right\} \cong \mathbb{R%
}$ and consider the \emph{augmented} chain complex $\widetilde{\mathcal{R}}%
=\left\{ \mathcal{R}_{p}\right\} _{p\geq -1}$ where the operator $\partial $
still satisfies $\partial ^{2}=0.$ Consequently, we obtain also the \emph{%
augmented} chain complex $\widetilde{\Omega }=\left\{ \Omega _{p}\right\}
_{p\geq -1}$ of $\partial $-invariant paths, where $\Omega _{p}$ with $%
\,p\geq 0$ is as before and $\Omega _{-1}=\mathcal{R}_{-1},$ as well as the 
\emph{reduced} homology groups $\{\widetilde{H}_{p}\}_{p\geq -1}$ where $%
\widetilde{H}_{-1}=\left\{ 0\right\} $, $\widetilde{H}_{p}=H_{p}$ for $p\geq
1$ and $H_{0}\cong \widetilde{H}_{0}\oplus \mathbb{R}.$ The reduced Euler
characteristic is%
\begin{equation}
\widetilde{\chi }\left( P\right) =\sum_{p\geq -1}\left( -1\right) ^{p}\dim
\Omega _{p}=\sum_{p\geq -1}\left( -1\right) ^{p}\dim \widetilde{H}_{p}=\chi
\left( P\right) -1.  \label{Er}
\end{equation}

Denote by $\left\langle ,\right\rangle $ the standard inner product in $%
\mathcal{R}_{p}$ defined by (\ref{ii}) in the case $p\geq 0$ and by $%
\left\langle e,e\right\rangle =1$ for $p=-1.$ If $u\in \mathcal{R}_{p}$ and $%
v\in \mathcal{R}_{q}$ with $p\neq q$ then set $\left\langle u,v\right\rangle
=0.$ As before, we denote by $\iota $ the standard inner product structure
in $\widetilde{\mathcal{R}}$.

\subsection{Join of path complexes and digraphs}

\label{SecJD}Let $V$ be a finite set.

\begin{definition}
For any paths $u\in \mathcal{R}_{p}\left( V\right) $ and $v\in \mathcal{R}%
_{q}\left( V\right) $ with $p,q\geq -1$ define their \emph{join} $u\cdot
v\in \mathcal{R}_{p+q+1}\left( V\right) $ as follows: first set 
\begin{equation*}
e_{i_{0}...i_{p}}\cdot e_{j_{0}...j_{q}}=e_{i_{0}...i_{p}j_{0}...j_{q}}
\end{equation*}%
and then extend this definition by linearity.
\end{definition}

In particular, we have $e_{i_{0}...i_{p}}\cdot e=e_{i_{0}...i_{p}}.$ The
following product formula holds for the chain complex\ $\widetilde{\mathcal{R%
}}\left( V\right) $: 
\begin{equation}
\partial \left( u\cdot v\right) =\left( \partial u\right) \cdot v+\left(
-1\right) ^{p+1}u\cdot \left( \partial v\right) .  \label{duv}
\end{equation}%
(see \cite[Lemma 2.2]{Grigoryan-Lin-Muranov-Yau2020} and \cite[Lemma 2.4]%
{GMY}).

Let $X,Y$ be two finite disjoint sets, set $Z=X\sqcup Y.$ Then all paths on $%
X$ and $Y$ can be considered as paths on $Z$. It follows easily from
definition of $u\cdot v$ that 
\begin{equation*}
u\in \mathcal{R}_{p}\left( X\right) \ \text{and\ }v\in \mathcal{R}_{q}\left(
Y\right) \ \ \Rightarrow \ \ u\cdot v\in \mathcal{R}_{p+q+1}\left( Z\right) .
\end{equation*}%
Also, for the standard inner product $\langle ,\rangle $ given by (\ref{ii}%
), we have 
\begin{equation}
\langle u\cdot v,\varphi \cdot \psi \rangle =\langle u,\varphi \rangle
\langle v,\psi \rangle ,  \label{uvfipsi}
\end{equation}%
for all $u\in \mathcal{R}_{p}\left( X\right) $, $v\in \mathcal{R}_{q}\left(
Y\right) $, $\varphi \in \mathcal{R}_{p^{\prime }}\left( X\right) $ and $%
\psi \in \mathcal{R}_{q^{\prime }}\left( X\right) $ (see also \cite[Lemma
3.10]{GMY}).

Let us extend the property (\ref{t}) of the definition of a path complex $P$
also to $n$-paths with $n=0$, that is, we allow in (\ref{t}) also $n=0.$
Then necessarily the empty path $e$ belongs to $P.$

\begin{definition}
Let $P\left( X\right) $ and $P\left( Y\right) $ be two path complexes over
finite disjoint sets $X$ and $Y$, respectively. Define the chain complex $%
P\left( Z\right) $ over the set $Z=X\sqcup Y$ as follows: $P\left( Z\right) $
consists of all the paths of the form $\,u\cdot v$ where $u\in P\left(
X\right) $ and $v\in P\left( Y\right) .$ The path complex $P\left( Z\right) $
is called the \emph{join} of $P\left( X\right) ,P\left( Y\right) $ and is
denoted by $P\left( Z\right) =P\left( X\right) \ast P\left( Z\right) .$
\end{definition}

Clearly, $P\left( X\right) $ and $P\left( Y\right) $ are subsets of $P\left(
Z\right) $.

\begin{definition}
If $X$ and $Y$ are digraphs then define their \emph{join} as a digraph $%
Z=X\ast Y$ where the set of vertices is $X\sqcup Y$ and the set of arrows
consists of all arrows of $X$, all arrows of $Y$ as well as of all arrows of
the form $x\rightarrow y$ where $x\in X$ and $y\in Y.$
\end{definition}

It is easy to see that 
\begin{equation*}
P\left( X\ast Y\right) =P\left( X\right) \ast P\left( Y\right)
\end{equation*}%
so that the operation of join of digraphs is compatible with join of path
complexes (cf. \cite{Grigoryan-Lin-Muranov-Yau2020}).

It is clear from the definition of $P\left( Z\right) $ that 
\begin{equation*}
u\in \mathcal{A}_{p}\left( X\right) \ \text{and\ }v\in \mathcal{A}_{q}\left(
Y\right) \ \ \Rightarrow \ \ u\cdot v\in \mathcal{A}_{p+q+1}\left( Z\right) .
\end{equation*}
It follows from (\ref{duv}) that, for all $p,q\geq -1$, 
\begin{equation*}
u\in \Omega _{p}\left( X\right) \ \ \text{and\ \ }v\in \Omega _{q}\left(
Y\right) \ \ \Rightarrow \ \ u\cdot v\in \Omega _{p+q+1}\left( Z\right)
\end{equation*}%
(see \cite[Prop 5.4]{Grigoryan-Lin-Muranov-Yau2020}). Furthermore, the
following version of the K\"{u}nneth formula is true for join: for any $%
r\geq -1$ 
\begin{equation}
\Omega _{r}\left( Z\right) =\bigoplus_{\left\{ p,q\geq -1:p+q+1=r\right\}
}\Omega _{p}\left( X\right) \otimes \Omega _{q}\left( Y\right)  \label{Kr}
\end{equation}%
where $u\otimes v$ for $u\in \Omega _{p}\left( X\right) $ and $v\in \Omega
_{q}\left( Y\right) $ is identified with the element $u\cdot v$ of $\Omega
_{r}\left( Z\right) $ (see \cite[Thm 5.5]{Grigoryan-Lin-Muranov-Yau2020}).

As a consequence of (\ref{Kr}) we obtain that%
\begin{equation}
\widetilde{\chi }\left( Z\right) =-\widetilde{\chi }\left( X\right) 
\widetilde{\chi }\left( Y\right) ,  \label{EZ}
\end{equation}%
where the minus comes from the additional $1$ in $r=p+q+1$.

\subsection{Operators $\partial ^{\ast }$ and $\Delta $ on joins}

We always assume in what follows that all the spaces $\mathcal{R}_{p}$ under
consideration are endowed with the standard inner product (\ref{ii}), in
particular, (\ref{uvfipsi}) is satisfied.

\begin{lemma}
Let $u\in \Omega _{p}\left( X\right) $ and $v\in \Omega _{q}\left( Y\right)
. $ Then%
\begin{equation}
\partial ^{\ast }\left( u\cdot v\right) =\left( \partial ^{\ast }u\right)
\cdot v+\left( -1\right) ^{p+1}u\cdot \left( \partial ^{\ast }v\right) .
\label{djoin}
\end{equation}
\end{lemma}

\begin{proof}
By definition, we have, for any $w\in \Omega _{p+q+2}\left( Z\right) $ 
\begin{equation*}
\langle \partial ^{\ast }\left( u\cdot v\right) ,w\rangle =\langle u\cdot
v,\partial w\rangle
\end{equation*}%
Any $w\in \Omega _{\ast }\left( Z\right) $ admits a representation%
\begin{equation*}
w=\sum_{k}\varphi _{k}\cdot \psi _{k}
\end{equation*}%
where the sum is finite and%
\begin{equation*}
\varphi _{k}\in \Omega _{p_{k}}\left( X\right) \ \ \text{and }\psi _{k}\in
\Omega _{q_{k}}\left( Y\right)
\end{equation*}%
with $p_{k}+q_{k}+1=p+q+2$ (see \cite[Thm 5.1]{GMY} and \cite[Thm 5.15]%
{Grigoryan-Lin-Muranov-Yau2020}). Then we have using (\ref{uvfipsi}) 
\begin{eqnarray*}
\langle \partial ^{\ast }\left( u\cdot v\right) ,w\rangle &=&\langle u\cdot
v,\sum \partial \left( \varphi _{k}\cdot \psi _{k}\right) \rangle \\
&=&\langle u\cdot v,\sum \left( \partial \varphi _{k}\cdot \psi _{k}+\left(
-1\right) ^{p_{k}+1}\varphi _{k}\cdot \partial \psi _{k}\right) \rangle \\
&=&\sum \langle u\cdot v,\partial \varphi _{k}\cdot \psi _{k}\rangle +\left(
-1\right) ^{p_{k}+1}\langle u\cdot v,\varphi _{k}\cdot \partial \psi
_{k}\rangle \\
&=&\sum \langle u,\partial \varphi _{k}\rangle \langle v,\psi _{k}\rangle
+\left( -1\right) ^{p_{k}+1}\langle u,\varphi _{k}\rangle \langle v,\partial
\psi _{k}\rangle \\
&=&\sum \langle \partial ^{\ast }u,\varphi _{k}\rangle \langle v,\psi
_{k}\rangle +\left( -1\right) ^{p_{k}+1}\langle u,\varphi _{k}\rangle
\langle \partial ^{\ast }v,\psi _{k}\rangle .
\end{eqnarray*}%
Note that if $p_{k}\neq p$ then $\langle u,\varphi _{k}\rangle =0.$ Hence,
we can replace $p_{k}$ everywhere by $p$ and obtain%
\begin{eqnarray*}
\langle \partial ^{\ast }\left( u\cdot v\right) ,w\rangle &=&\sum \langle
\partial ^{\ast }u,\varphi _{k}\rangle \langle v,\psi _{k}\rangle +\left(
-1\right) ^{p+1}\langle u,\varphi _{k}\rangle \langle \partial ^{\ast
}v,\psi _{k}\rangle \\
&=&\sum \langle \partial ^{\ast }u\cdot v,\varphi _{k}\cdot \psi _{k}\rangle
+\langle \left( -1\right) ^{p+1}u\cdot \partial ^{\ast }v,\varphi _{k}\cdot
\psi _{k}\rangle \\
&=&\langle \partial ^{\ast }u\cdot v+\left( -1\right) ^{p+1}u\cdot \partial
^{\ast }v,\sum \varphi _{k}\cdot \psi _{k}\rangle \\
&=&\langle \partial ^{\ast }u\cdot v+\left( -1\right) ^{p+1}u\cdot \partial
^{\ast }v,w\rangle ,
\end{eqnarray*}%
whence (\ref{djoin}) follows.
\end{proof}

For the \emph{Hodge Laplacian}%
\begin{equation*}
\Delta u=\partial \partial ^{\ast }u+\partial ^{\ast }\partial u
\end{equation*}%
we have then the following identity.

\begin{lemma}
For all $u\in \Omega _{p}\left( X\right) $ and $v\in \Omega _{q}\left(
X\right) $ we have%
\begin{equation}
\Delta \left( u\cdot v\right) =\left( \Delta u\right) \cdot v+u\cdot \Delta
v.  \label{Deltajoin}
\end{equation}
\end{lemma}

\begin{proof}
Indeed, by (\ref{djoin}) we have%
\begin{eqnarray*}
\partial \partial ^{\ast }\left( u\cdot v\right) &=&\partial \left( \partial
^{\ast }u\cdot v+\left( -1\right) ^{p+1}u\cdot \partial ^{\ast }v\right) \\
&=&\partial \left( \partial ^{\ast }u\cdot v\right) +\left( -1\right)
^{p+1}\partial \left( u\cdot \partial ^{\ast }v\right) \\
&=&\partial \partial ^{\ast }u\cdot v+\left( -1\right) ^{p+2}\partial ^{\ast
}u\cdot \partial v \\
&&+\left( -1\right) ^{p+1}\left( \partial u\cdot \partial ^{\ast }v+\left(
-1\right) ^{p+1}u\cdot \partial \partial ^{\ast }v\right) \\
&=&\partial \partial ^{\ast }u\cdot v+\left( -1\right) ^{p}\partial ^{\ast
}u\cdot \partial v+\left( -1\right) ^{p+1}\partial u\cdot \partial ^{\ast
}v+u\cdot \partial \partial ^{\ast }v
\end{eqnarray*}%
and by (\ref{duv})%
\begin{eqnarray*}
\partial ^{\ast }\partial \left( u\cdot v\right) &=&\partial ^{\ast }\left(
\partial u\cdot v+\left( -1\right) ^{p+1}u\cdot \partial v\right) \\
&=&\partial ^{\ast }\left( \partial u\cdot v\right) +\left( -1\right)
^{p+1}\partial ^{\ast }\left( u\cdot \partial v\right) \\
&=&\partial ^{\ast }\partial u\cdot v+\left( -1\right) ^{p}\partial u\cdot
\partial ^{\ast }v \\
&&+\left( -1\right) ^{p+1}\left( \partial ^{\ast }u\cdot \partial v+\left(
-1\right) ^{p+1}u\cdot \partial ^{\ast }\partial v\right) \\
&=&\partial ^{\ast }\partial u\cdot v+\left( -1\right) ^{p}\partial u\cdot
\partial ^{\ast }v+\left( -1\right) ^{p+1}\partial ^{\ast }u\cdot \partial
v+u\cdot \partial ^{\ast }\partial v.
\end{eqnarray*}%
Adding up the two identities, we see that the terms $\partial ^{\ast }u\cdot
\partial v$ and $\partial u\cdot \partial ^{\ast }v$ cancel out, and we
obtain (\ref{Deltajoin}).
\end{proof}

\subsection{Torsion of joins}

\label{SecTJ}Let $P$ be a path complex over a set $V$ with the standard
inner product structure $\iota $ given by (\ref{ii}). By means of the
augmented chain complex $\widetilde{\Omega }\left( P\right) $, let us define
the \emph{reduced} analytic torsion $\widetilde{T}\left( P\right) $ by%
\begin{equation*}
\log \widetilde{T}\left( P\right) =\log T(\widetilde{\Omega }\left( P\right)
,\iota )=\frac{1}{2}\sum_{p=-1}^{N}(-1)^{p}\,p\,\zeta _{p}^{\prime }(0).
\end{equation*}%
In the previous sections we used the standard analytic torsion $T\left(
P\right) $ given by%
\begin{equation*}
\log T\left( P\right) =\log T(\Omega \left( P\right) ,\iota )=\frac{1}{2}%
\sum_{p=0}^{N}(-1)^{p}\,p\,\zeta _{p}^{\prime }(0)
\end{equation*}%
The relation between $T\left( P\right) $ and $\widetilde{T}\left( P\right) $
is given by the following formula.

\begin{lemma}
We have 
\begin{equation}
T\left( P\right) =\sqrt{\left\vert V\right\vert }\widetilde{T}\left(
P\right) .  \label{rtM}
\end{equation}
\end{lemma}

\begin{proof}
The zeta function $\zeta _{p}\left( s\right) $ is determined by the operator 
$\Delta _{p}$ that is the same for the chain complexes $\Omega $ and $%
\widetilde{\Omega }$ for all $p\geq 1.$ For $p=0$ the operators $\Delta _{p}$
are different for these two complexes, but the value $p=0$ does not give any
contribution to the analytic torsions. Hence, the difference is determined
by $p=-1,$ that is, 
\begin{equation}
\log \widetilde{T}\left( P\right) =\log T\left( P\right) +\frac{1}{2}\,\zeta
_{-1}^{\prime }(0).  \label{zeta-1}
\end{equation}%
For $e\in \Omega _{-1}$ we have $\partial e=0$ and 
\begin{equation*}
\partial ^{\ast }e=\sum_{i\in V}e_{i},
\end{equation*}
because for any $i$ 
\begin{equation*}
\left\langle \partial ^{\ast }e,e_{i}\right\rangle =\left\langle e,\partial
e_{i}\right\rangle =\left\langle e,e\right\rangle =1.
\end{equation*}%
Hence, 
\begin{equation*}
\Delta e=\partial \partial ^{\ast }e+\partial ^{\ast }\partial e=\partial
\sum_{i\in V}e_{i}=\left\vert V\right\vert e.
\end{equation*}%
Therefore, $\zeta _{-1}\left( s\right) =\left\vert V\right\vert ^{-s}$ and $%
\zeta _{-1}^{\prime }(0)=-\log \left\vert V\right\vert $. Substituting into (%
\ref{zeta-1}) we obtain%
\begin{equation}
\log \widetilde{T}\left( P\right) =\log T\left( P\right) -\frac{1}{2}\log
\left\vert V\right\vert ,  \label{T-T}
\end{equation}%
which is equivalent to (\ref{rtM}).
\end{proof}

The next theorem is our main result about torsion on joins.

\begin{theorem}
\label{Tmain2}For the join path complex $P\left( Z\right) =P\left( X\right)
\ast P\left( Y\right) $ we have 
\begin{equation}
\log \widetilde{T}(Z)=-\widetilde{\chi }(Y)\log \widetilde{T}(X)-\widetilde{%
\chi }(X)\log \widetilde{T}(Y)  \label{T1-}
\end{equation}%
where $\widetilde{\chi }$ is the reduced Euler characteristic.
\end{theorem}

\begin{proof}[Proof of Theorem \protect\ref{Tmain2}]
The proof is similar to that of Theorem \ref{T:main_thm1}. Let $E_{p}\left(
\lambda \right) $ be the eigenspace of $\Delta _{p}$ with the eigenvalue $%
\lambda $, and set 
\begin{equation*}
n_{p}\left( \lambda \right) =\dim E_{p}\left( \lambda \right) .
\end{equation*}%
Lemmas \ref{LemE+} and \ref{Lemiso} go unchanged also for the augmented
chain complex $\widetilde{\Omega }=\{\Omega _{p}\}_{p\geq -1}.$ The same
argument as in the proof of Lemma \ref{LemNp} gives for any $\lambda >0$
that 
\begin{equation}
\sum\limits_{p=-1}^{N}(-1)^{p}n_{p}(\lambda )=0,  \label{sumnp}
\end{equation}%
because $n_{-1}^{\prime \prime }\left( \lambda \right) =0$: indeed, for any $%
\varphi \in E_{-1}^{\prime \prime }\left( \lambda \right) $ we have $%
\partial \varphi =0$ and, hence, $\varphi =\frac{1}{\lambda }\partial ^{\ast
}\partial \varphi =0.$

Arguing as in the proof of Theorem \ref{T:main_thm1} and using the K\"{u}%
nneth formula (\ref{Kr}) we obtain in place of (\ref{zetan}) the following
identity: 
\begin{equation*}
\zeta _{r,Z}(s)=\sum_{\lambda +\mu >0}\sum_{p+q+1=r}(\lambda +\mu
)^{-s}n_{p}(\lambda ,X)n_{q}(\mu ,Y)
\end{equation*}%
for any $r\geq -1.$ It follows that%
\begin{eqnarray}
&&\sum_{r\geq -1}(-1)^{r}r\zeta _{r,Z}(s)  \notag \\
&=&\sum_{\lambda +\mu >0}(\lambda +\mu )^{-s}\sum_{p\geq -1}\sum_{q\geq
-1}(-1)^{p+q+1}(p+q+1)n_{p}(\lambda ,X)n_{q}(\mu ,Y)  \notag \\
&=&-\sum_{\lambda +\mu >0}(\lambda +\mu )^{-s}\left( \sum_{p\geq
-1}(-1)^{p}p\,n_{p}(\lambda ,X)\right) \left( \sum_{q\geq
-1}(-1)^{q}n_{q}(\mu ,Y)\right)  \label{2b} \\
&&-\sum_{\lambda +\mu >0}(\lambda +\mu )^{-s}\left( \sum_{p\geq
-1}(-1)^{p}n_{p}(\lambda ,X)\right) \left( \sum_{q\geq
-1}(-1)^{q}q\,n_{q}(\mu ,Y)\right)  \label{3b} \\
&&-\sum_{\lambda +\mu >0}(\lambda +\mu )^{-s}\sum_{p\geq -1}\sum_{q\geq
-1}(-1)^{p+q}n_{p}(\lambda ,X)n_{q}(\mu ,Y).  \label{4b}
\end{eqnarray}%
By (\ref{sumnp}), if $\lambda >0$ then 
\begin{equation*}
\sum_{p\geq -1}(-1)^{p}n_{p}(\lambda ,X)=0
\end{equation*}%
and if $\mu >0$ then%
\begin{equation*}
\sum_{q\geq -1}(-1)^{q}n_{q}(\mu ,Y)=0.
\end{equation*}%
Hence, the double sum in (\ref{4b}) is equal to zero, while in (\ref{2b})-(%
\ref{3b}) all the terms with $\lambda >0$ and $\mu >0$ vanish. We obtain%
\begin{eqnarray*}
\sum_{r\geq -1}(-1)^{r}r\zeta _{r,Z}(s) &=&-\sum_{\lambda >0,\mu =0}\lambda
^{-s}\left( \sum_{p\geq -1}(-1)^{p}p\,n_{p}(\lambda ,X)\right) \left(
\sum_{q\geq -1}(-1)^{q}n_{q}(0,Y)\right) \\
&&-\sum_{\mu >0,\lambda =0}\mu ^{-s}\left( \sum_{p\geq
-1}(-1)^{p}n_{p}(0,X)\right) \left( \sum_{q\geq -1}(-1)^{q}q\,n_{q}(\mu
,Y)\right) \\
&=&-\widetilde{\chi }(Y)\sum_{p\geq -1}(-1)^{p}p\zeta _{p,X}(s)-\widetilde{%
\chi }(X)\sum_{q\geq -1}(-1)^{q}q\zeta _{q,Y}(s).
\end{eqnarray*}%
Taking derivative of the both sides at $s=0$ and using the definition of
analytic torsion, we obtain (\ref{T1-}).
\end{proof}

For the standard analytic torsion $T$ we obtain the following.

\begin{corollary}
\label{Cor2}We have 
\begin{eqnarray}
\log T(Z) &=&-\widetilde{\chi }(Y)\log T(X)-\widetilde{\chi }(X)\log T(Y) 
\notag \\
&&+\frac{1}{2}\log \left\vert Z\right\vert +\frac{\widetilde{\chi }\left(
Y\right) }{2}\log \left\vert X\right\vert +\frac{\widetilde{\chi }\left(
X\right) }{2}\log \left\vert Y\right\vert .  \label{Tj}
\end{eqnarray}
\end{corollary}

\begin{proof}
Using (\ref{rtM}) and (\ref{Er}), we obtain%
\begin{eqnarray*}
\log T\left( Z\right) &=&\frac{1}{2}\log \left\vert Z\right\vert +\log 
\widetilde{T}\left( Z\right) \\
&=&\frac{1}{2}\log \left\vert Z\right\vert -\left( \widetilde{\chi }(Y)\log 
\widetilde{T}(X)+\widetilde{\chi }(X)\log \widetilde{T}(Y)\right) \\
&=&\frac{1}{2}\log \left\vert Z\right\vert -\widetilde{\chi }(Y)\left( \log
T(X)-\frac{1}{2}\log \left\vert X\right\vert \right) -\widetilde{\chi }%
(X)\left( \log T(Y)-\frac{1}{2}\log \left\vert Y\right\vert \right) \\
&=&\frac{1}{2}\log \left\vert Z\right\vert +\widetilde{\chi }(Y)\frac{1}{2}%
\log \left\vert X\right\vert +\widetilde{\chi }(X)\frac{1}{2}\log \left\vert
Y\right\vert \\
&&-\widetilde{\chi }(Y)\log T(X)-\widetilde{\chi }(X)\log T(Y).
\end{eqnarray*}
\end{proof}

For any $n\geq 2$ define on the set $X^{\ast n}=\underset{n\ \mathrm{times}}{%
\underbrace{X\sqcup ...\sqcup X}}$ the following path complex 
\begin{equation*}
P\left( X^{\ast n}\right) =\underset{n\ \mathrm{times}}{\underbrace{P\left(
X\right) \ast ...\ast P\left( X\right) }}=P(X)^{\ast n}.
\end{equation*}

\begin{corollary}
\label{Cor2n}We have 
\begin{equation}
\log \widetilde{T}(X^{\ast n})=n\left( -\widetilde{\chi }\left( X\right)
\right) ^{n-1}\log \widetilde{T}\left( X\right)  \label{TtiXn}
\end{equation}%
and%
\begin{equation}
\log T(X^{\ast n})=n\left( 1-\chi \left( X\right) \right) ^{n-1}\log T\left(
X\right) -\frac{1}{2}n\left( 1-\chi \left( X\right) \right) ^{n-1}\log
\left\vert X\right\vert +\frac{1}{2}\log \left( n\left\vert X\right\vert
\right) .  \label{TX*n}
\end{equation}
\end{corollary}

\begin{proof}
Denote $\log \widetilde{T}(X^{\ast n})=x_{n}$ and $-\widetilde{\chi }\left(
X\right) =a.$ Then $-\widetilde{\chi }(X^{\ast n})=a^{n}$, and we have by (%
\ref{T1-}) 
\begin{equation*}
x_{n+1}=ax_{n}+a^{n}x_{1}.
\end{equation*}%
By induction we obtain $x_{n}=na^{n-1}x_{1}$, which proves (\ref{TtiXn}).

Using (\ref{rtM}) (or (\ref{T-T})) and (\ref{Er}), we obtain from (\ref%
{TtiXn})%
\begin{eqnarray*}
\log T(X^{\ast n}) &=&\log \widetilde{T}(X^{\ast n})+\frac{1}{2}\log
\left\vert X^{\ast n}\right\vert \\
&=&n\left( 1-\chi \left( X\right) \right) ^{n-1}(\log T\left( X\right) -%
\frac{1}{2}\log \left\vert X\right\vert )+\frac{1}{2}\log \left( n\left\vert
X\right\vert \right) \\
&=&n\left( 1-\chi \left( X\right) \right) ^{n-1}\log T\left( X\right) -\frac{%
1}{2}n\left( 1-\chi \left( X\right) \right) ^{n-1}\log \left\vert
X\right\vert +\frac{1}{2}\log \left( n\left\vert X\right\vert \right) ,
\end{eqnarray*}%
which proves (\ref{TX*n}).
\end{proof}

For example, if $\chi \left( X\right) =0$ then (\ref{TX*n}) yields%
\begin{equation}
\log T(X^{\ast n})=n\log T\left( X\right) -\frac{n-1}{2}\log \left\vert
X\right\vert +\frac{1}{2}\log n,  \label{TX0}
\end{equation}%
if $\chi \left( X\right) =1$ then (\ref{TX*n}) yields 
\begin{equation}
T(X^{\ast n})=\sqrt{n\left\vert X\right\vert },  \label{TX1}
\end{equation}%
and if $\chi \left( X\right) =2$ then%
\begin{equation}
\log T(X^{\ast n})=n\left( -1\right) ^{n-1}\log T\left( X\right) +\frac{%
1+n\left( -1\right) ^{n}}{2}\log \left\vert X\right\vert +\frac{1}{2}\log n.
\label{TX2}
\end{equation}

\begin{example}
\label{ExSim}Let $O=\left\{ \bullet \right\} $ be a trivial digraph of one
vertex. It is easy to see that $\chi \left( O\right) =1.$ The join $O\ast
O=O^{\ast 2}$ is the interval $I$ from Example \ref{ExI}, the join $O\ast
I=O^{\ast 3}$ is the triangle from Example \ref{Ex1}. More generally, $%
O^{\ast n}$ can be regarded as an $\left( n-1\right) $-dimensional digraph 
\emph{simplex}\ (see Fig. \ref{pic11}).\FRAME{ftbhFU}{3.7714cm}{3.2735cm}{0pt%
}{\Qcb{Simplex $O^{\ast 4}=I^{\ast 2}$}}{\Qlb{pic11}}{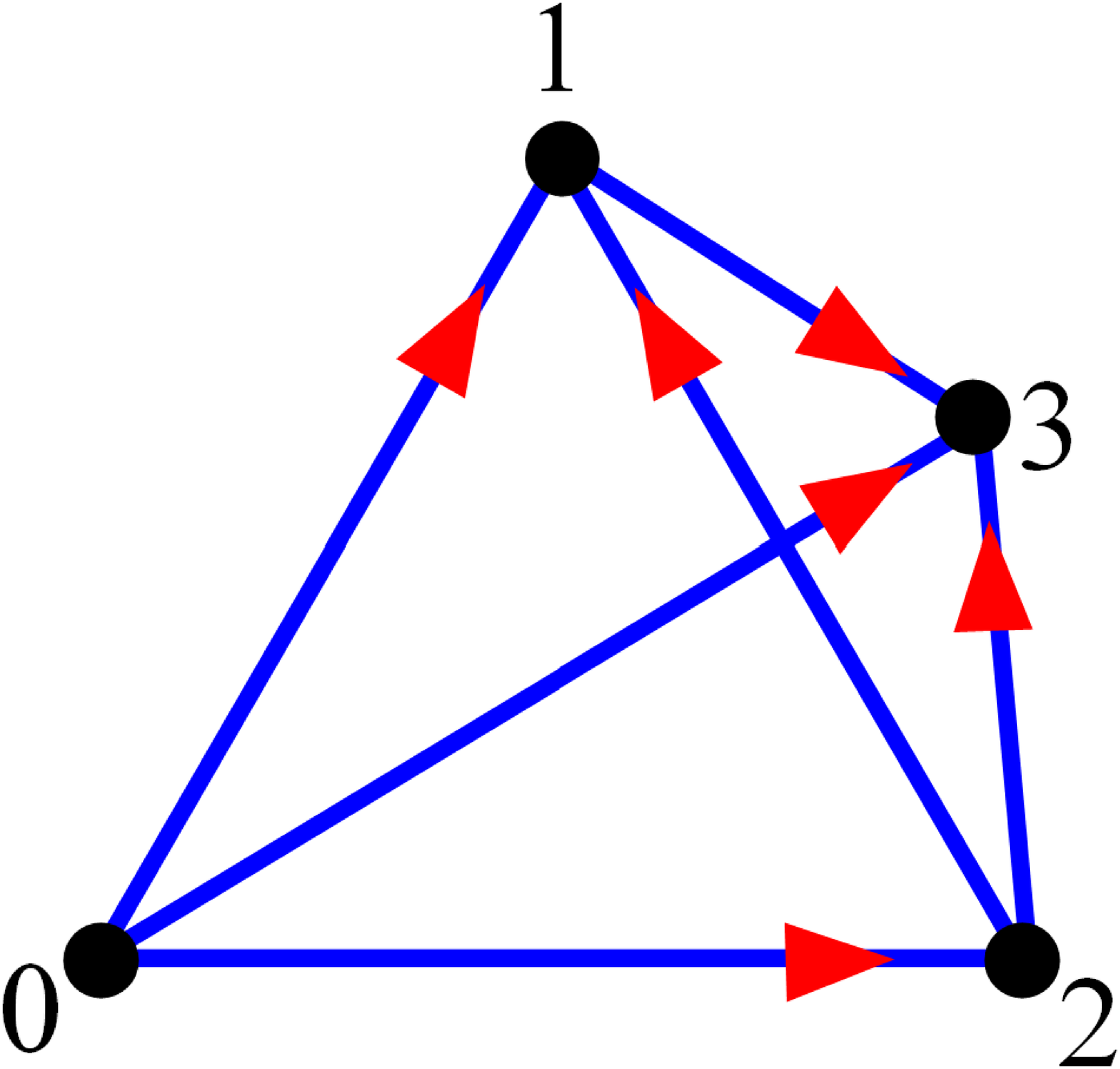}{\special%
{language "Scientific Word";type "GRAPHIC";maintain-aspect-ratio
TRUE;display "USEDEF";valid_file "F";width 3.7714cm;height 3.2735cm;depth
0pt;original-width 13.8938cm;original-height 12.0487cm;cropleft "0";croptop
"1";cropright "1";cropbottom "0";filename 'pic11.eps';file-properties
"XNPEU";}}

From (\ref{TX1}) we obtain that 
\begin{equation*}
T(O^{\ast n})=\sqrt{n}.
\end{equation*}%
For example, $T(O^{\ast 3})=\sqrt{3}$ as we have seen in Example \ref{Ex1}.
\end{example}

\begin{example}
\label{Exsphere}Consider the digraph $D=\left\{ \bullet ,\bullet \right\} $
consisting of two disjoint vertices. The join $D^{\ast 2}$ is a
quadrilateral, and $D^{\ast 3}$ is an octahedron (see Fig. \ref{pic13}). The
digraph $D^{\ast n}$ can be regarded as a digraph analogue of an $\left(
n-1\right) $-dimensional sphere. \FRAME{ftbhFU}{9.2415cm}{3.8683cm}{0pt}{%
\Qcb{The octahedron $D^{\ast 3}$ is shown in two ways. The green subgraph is 
$D^{\ast 2}.$}}{\Qlb{pic13}}{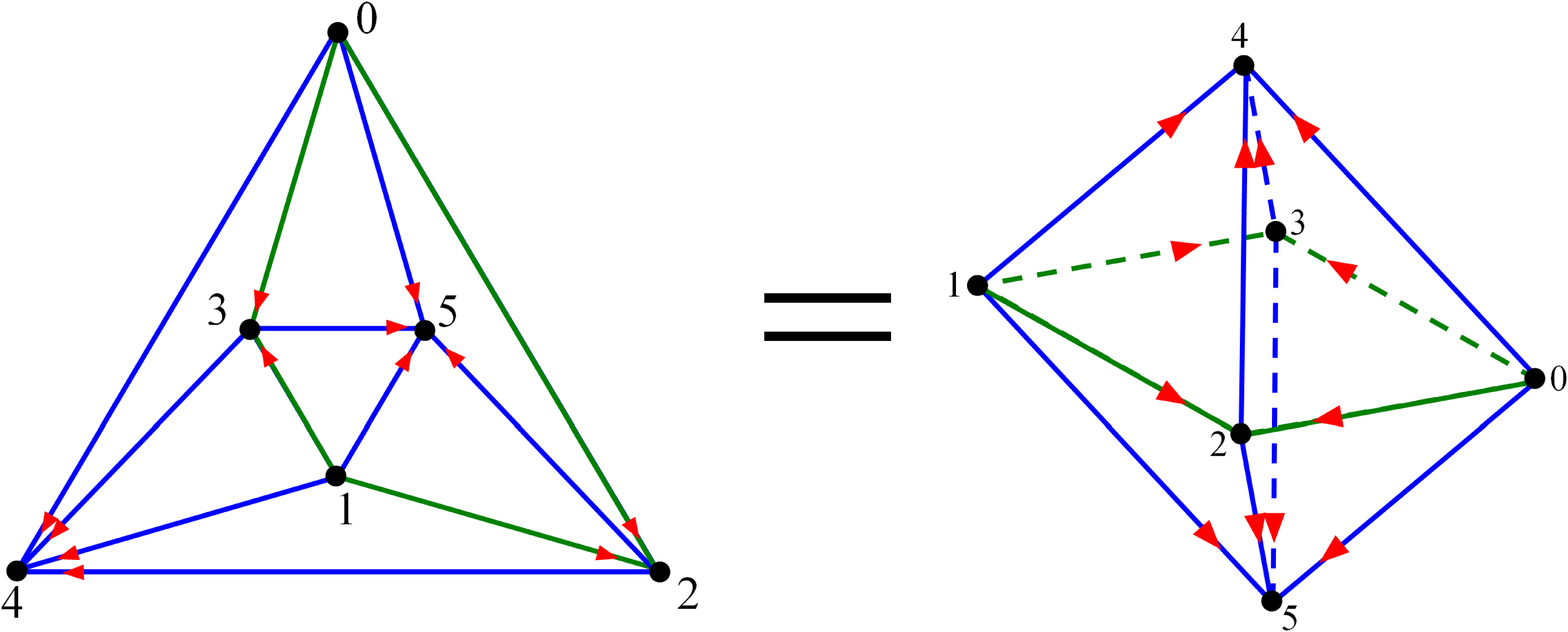}{\special{language "Scientific
Word";type "GRAPHIC";maintain-aspect-ratio TRUE;display "USEDEF";valid_file
"F";width 9.2415cm;height 3.8683cm;depth 0pt;original-width
35.5671cm;original-height 14.8356cm;cropleft "0";croptop "1";cropright
"1";cropbottom "0";filename 'pic13.eps';file-properties "XNPEU";}}

We have $\chi \left( D\right) =2$ and $\tau \left( D\right) =T\left(
D\right) =1.$ By (\ref{TX2}) we obtain that 
\begin{equation*}
\log T((D^{\ast n})=\frac{1+n\left( -1\right) ^{n}}{2}\log 2+\frac{1}{2}\log
n,
\end{equation*}%
that is%
\begin{equation*}
T(D^{\ast n})=\sqrt{n}2^{\frac{1+n\left( -1\right) ^{n}}{2}}.
\end{equation*}%
For example, 
\begin{equation*}
T(D^{\ast 2})=\sqrt{2}2^{\frac{3}{2}}=4
\end{equation*}%
and%
\begin{equation*}
T(D^{\ast 3})=\sqrt{3}2^{-1}=\frac{\sqrt{3}}{2}.
\end{equation*}
\end{example}

\begin{example}
Let $G$ be a cyclic digraph from Example \ref{Excycle} with $m$ vertices.
Since $\chi \left( G\right) =0$ and $T\left( G\right) =m$, we obtain by (\ref%
{TX0})%
\begin{equation*}
\log T(G^{\ast n})=n\log m-\frac{n-1}{2}\log m+\frac{1}{2}\log n
\end{equation*}%
and 
\begin{equation*}
T(G^{\ast n})=\sqrt{n}m^{\frac{n+1}{2}}.
\end{equation*}
\end{example}

\begin{acknowledgement}
\medskip We would like to thank Wang Chong for the computation of torsions
of digraphs in some Examples.
\end{acknowledgement}

\

\end{document}